\documentclass[11pt,a4paper]{amsart}
\usepackage{amssymb,amsmath,stmaryrd,pstricks,color}

\newtheorem{lemma}{Lemma}
\newtheorem*{corollary}{Corollary}
\newtheorem{proposition}{Proposition}

\newtheorem*{mytheorem1}{\bf Theorem 1}
\newtheorem*{mytheorem1'}{\bf Theorem 1$^\prime$}
\newtheorem*{mytheorem2}{\bf Theorem 2}
\newtheorem*{mytheorem2'}{\bf Theorem 2$^\prime$}
\newtheorem*{mytheorem3}{\bf Theorem 3}
\newtheorem*{mytheorem3'}{\bf Theorem 3$^\prime$}
\newtheorem*{mytheorem3''}{\bf Theorem 3$^{\prime\prime}$}
\newtheorem*{mytheorem4}{\bf Theorem 4}
\newtheorem*{mytheorem4'}{\bf Theorem 4$^\prime$}
\newtheorem*{mytheorem5}{\bf Theorem 5}
\newtheorem*{mytheorem6}{\bf Theorem 6}
\newtheorem*{mytheorem7}{\bf Theorem 7}
\newtheorem*{mytheorem8}{\bf Theorem 8}

\newtheorem*{mytheorem9}{\bf Theorem 9}
\newtheorem*{mytheorem10}{\bf Theorem 10}
\newtheorem*{mytheorem11}{\bf Theorem 11}
\newtheorem*{mytheorem12}{\bf Theorem 12}
\newtheorem*{mytheorem13}{\bf Theorem 13}

\newcommand{\ilim}{\mathop{\varprojlim}\limits}

\def\GL{\operatorname{\mathbf{GL}}}

\def\O{\operatorname{\mathbf{O}}}
\def\U{\operatorname{\mathbf{U}}}
\def\SO{\operatorname{\mathbf{SO}}}
\def\SL{\operatorname{\mathbf{SL}}}
\def\Sp{\operatorname{\mathbf{Sp}}}

\def\Aut{\operatorname{Aut}}

\def\Tr{\operatorname{Tr}}

\def\mod{\operatorname{mod}}
\def\ad{\operatorname{ad}}

\def\Gal{\operatorname{Gal}}

\def\Hom{\operatorname{Hom}}
\def\Im{\operatorname{Im}}
\def\car{\operatorname{char}}

\def\den{\operatorname{den}}
\def\disc{\operatorname{disc}}

\def\End{\operatorname{End}}

\def\Q{\mathbf Q}
\def\Z{\mathbf Z}

\def\H{\mathbf H}
\def\C{\mathbf C}
\def\R{\mathbf R}
\def\F{\mathbf F}
\def\P{\mathbf P}
\def\G{\mathbf G}

\def\L{\mathbf L}
\def\M{\mathbf M}

\def\<{\langle}
\def\>{\rangle}

\title{Bounds for the orders of the finite subgroups of $G(k)$}
\author{Jean-Pierre Serre}

\begin{document}

\markright{INTRODUCTION}
%\markleft{J.-P. SERRE, ORDERS OF FINITE SUBGROUPS}
\begin{center}
{\small Group Representation Theory, eds. M. Geck, D. Testerman, J. Th\'evenaz,}

{\small EPFL Press, Lausanne 2007, pp. 405-450 }\vskip1cm
{\Large Bounds for the orders of the finite subgroups of $G(k)$}

\vskip 0.7cm
Jean-Pierre SERRE
\end{center} 
\vskip1cm
\specialsection*{\bf Introduction}

The present text reproduces - with a number of additions - a series of three two-hour lectures given at the Ecole Polytechnique F\'ed\'erale de Lausanne (E.P.F.L.) on May 25-26-27, 2005.

The starting point is a classical result of Minkowski, dating from 1887, which gives a multiplicative upper bound for the orders of the finite subgroups of $\GL_n(\Q)$. The method can easily be extended to other algebraic groups than $\GL_n$, and the field $\Q$ can be replaced by any number field. What is less obvious is that:

a) one can work over an arbitrary ground field;

b) in most cases one may construct examples showing that the bound thus obtained is optimal.

This is what I explain in the lectures.
\vskip 0.3cm

Lecture I is historical: Minkowski (\S 1), Schur (\S 2), Blichfeldt and others (\S 3). The results it describes are mostly well-known, so that I did not feel compelled to give complete proofs. 

Lecture II gives upper bounds for the order of a finite $\ell$-subgroup of $G(k)$, where $G$ is a reductive group over a field $k$, and $\ell$ is a prime number. These bounds depend on $G$ via its root system, and on $k$ via the size of the Galois group of its $\ell$-cyclotomic tower (\S 4). One of these bounds (called here the S-bound, cf. \S 5) is a bit crude but is easy to prove and to apply. The second one (called the M-bound) is the most interesting one (\S 6). Its proof follows Minkowski's method, combined with Chebotarev's density theorem (for schemes of any dimension, not merely dimension 1); it has a curious cohomological generalization cf. \S 6.8. The last subsection (\S 6.9) mentions some related problems, not on semisimple groups, but on Cremona groups; for instance: does the field $\Q(X,Y,Z)$ have an automorphism of order 11 ? 

Lecture III gives the construction of ``optimal" large subgroups. The case of the classical groups (\S 9) is not difficult. Exceptional groups such as $E_8$ are a different matter; to handle them, we shall use Galois twists, braid groups and Tits groups, cf. \S\S 10-12.
\vfill\eject
{\sl Acknowledgements.}  A first draft of these notes, made by D. Testerman and R. Corran, has been very useful; and so has been the generous help of D. Testerman with the successive versions of the text. My thanks go to both of them, and to the E.P.F.L. staff for its hospitality. I also thank M. Brou\'e and J. Michel for several discussions on braid groups.

\vskip 0.5cm
\noindent J-P. Serre\hskip 2cm April 2006
\vskip2cm
\begin{center}
{\bf Table of Contents}
\end{center}
\vskip 1cm
${}$\hskip2cm Lecture I. History: Minkowski, Schur, ...
\begin{enumerate}
\item[1.] Minkowski
\item[2.] Schur
\item[3.] Blichfeldt and others
\end{enumerate}
${}$\hskip2cm Lecture II. Upper bounds
\begin{enumerate}
\item[4.] The invariants $t$ and $m$
\item[5.] The S-bound
\item[6.] The M-bound
\end{enumerate}
${}$\hskip2cm Lecture III. Construction of large subgroups
\begin{enumerate}
\item[7.] Statements
\item[8.] Arithmetic methods $(k = \Q)$
\item[9.] Proof of theorem 9 for classical groups
\item[10.] Galois twists
\item[11.] A general construction
\item[12.] Proof of theorem 9 for exceptional groups
\item[13.] Proof of theorems 10 and 11
\item[14.] The case $m = \infty$
\end{enumerate}
${}$\hskip2cm References

%%%%%%%%%%%%%%%%%%%%%%%%%%%%%%%%%%%%%%%%%%%%%%%%%%%%%%%%%%%%%%%%%%%%%%%
%%%%%%%%%%%%%%%%%%%%%%%%%%%%%%%%%%%%%%%%%%%%%%%%%%%%%%%%%%%%%%%%%%%%%%

%\include{serredonna_1}
%\markboth{J.-P. SERRE, ORDERS OF FINITE SUBGROUPS}{bidon}

\markright{LECTURE I:  HISTORY: MINKOWSKI, SCHUR, ...}

%\markleft{J.-P. SERRE, ORDERS OF FINITE SUBGROUPS}
\specialsection*{\bf I.  History: Minkowski, Schur, ...}

\vskip 0.5cm
\begin{center}
{\bf {\S 1. Minkowski}}
\end{center}
\vskip 0.3cm

Reference: [Mi 87].

\setcounter{section}{1}
\subsection{Statements} We shall use the following notation:

$\ell$ is a fixed prime number; when we need other primes we usually denote
 them by $p$; 

the $\ell$-adic valuation of a rational number $x$ is denoted by $v_\ell (x)$; one has $v_{\ell}(\ell ) = 1$, and $v_\ell (x) = 0$ if $x$ is an integer with $(x, \ell ) = 1$; 

the number of elements of a finite set $A$  is denoted by $|A|$; we write $v_\ell (A)$ instead of $v_\ell (|A|)$; if $A$ is a group, $\ell^{^{v_\ell (A)}}$ is the order of an $\ell$-Sylow of $A$; 

if $x$ is a real number, its integral part (``floor") is denoted by $[x]$. 

\medskip

\noindent We may now state Minkowski's theorem ([Mi 87]):

\begin{mytheorem1}
\label{I.1}
Let $n$ be an integer $\ge 1$, and let $\ell$ be a prime number. Define{\rm :}
$$
M(n,\ell ) = \left[ \frac{n}{\ell-1} \right] + \left[ \frac{n}{\ell(\ell-1)} \right] + \left[ \frac{n}{\ell^2(\ell-1)} \right] + \cdots
$$
Then{\rm :}

\noindent {\rm (i)} If $A$ is a finite subgroup of $\GL_n (\Q )$, we have $v_\ell (A) \le M(n,\ell ).$ 

\noindent {\rm (ii)} There exists a finite $\ell$-subgroup $A$ of  $\GL_n(\Q )$ with $v_\ell (A) = M(n,\ell )$.
\end{mytheorem1}

The proof will be given in \S 1.3 and \S 1.4.
\vskip 0.2cm
\noindent {\sl Remarks.}

1) Let us define an integer $M(n)$ by:
$$
M(n) = \prod_\ell\, \ell^{M(n,\ell)} .
$$
Part (i) of th.1 says that the order of any finite subgroup of $\GL_n (\Q)$ {\sl divides} $M(n)$, and part (ii) says that $M(n)$ is the smallest integer having  this property. Hence $M(n)$ is a sharp multiplicative bound for $|A|$. 

Here are the values of $M(n)$ for $n \le 8$:
\vskip 0.1cm
\noindent $M(1) = 2$

\noindent $M(2) = 2^3\!\cdot3=24$

\noindent $M(3) = 2^4\!\cdot3 = 48$

\noindent $M(4) = 2^7\!\cdot3^2\!\cdot5 = 5760$

\noindent $M(5) = 2^8\!\cdot3^2\!\cdot 5 = 11520$

\noindent $M(6) = 2^{10}\!\cdot3^4\!\cdot5\cdot 7 = 2903040 $

\noindent $M(7) = 2^{11}\!\cdot3^4\!\cdot5\cdot7 = 5806080$

\noindent $M(8) = 2^{15}\!\cdot3^5\!\cdot5^2\!\cdot 7 = 1393459200.$
\vskip 0.1cm
Note that 
$$M(n)/M(n-1) =\left\{ \begin{array}{ll}
2 &{\hbox{\sl if }} n {\hbox{\sl { is odd }}}\\
&\\
{\hbox{\rm {denominator of }}} b_n/n& {\hbox{\sl { if }}} n {\hbox{\sl { is even}}},\end{array}\right.$$

\noindent where $b_n$ is the $n$-th Bernoulli number. (The occurence of the Bernoulli numbers is natural in view of the mass formulae which Minkowski had proved a few years before.)

2) One may ask whether there is a finite subgroup $A$ of $\GL_n(\Q)$ of order $M(n)$. It is so for $n = 1$ and $n = 3$ and probably for no other value of $n$ (as Burnside already remarked on p.484 of [Bu 11]). Indeed, some incomplete arguments of Weisfeiler and Feit would imply that the upper bound of $|A|$ is $2^n\cdot n$! if $n > 10$, which is much smaller than $M(n)$. See the comments of Guralnick-Lorenz in [GL 06], \S 6.1.

\vskip 0.2cm
\noindent{\small{\sl Exercise.} Let $\left[\frac{n}{\ell-1}\right] = \sum a_i\ell^{i}, 0 \le a_i \le \ell -1$, be the $\ell$-adic expansion of $\left[\frac{n}{\ell-1}\right].$

\noindent Show that
$
M(n,\ell ) = \sum a_i \frac{\ell^{i+1}-1}{\ell -1} = \sum M(a_i\ell^{i} (\ell-1),\ell).
$}

\subsection{Minkowski's lemma.} Minkowski's paper starts with the following often quoted lemma:

\begin{lemma}
\label{lem1} If $m \ge 3$, the kernel of $\GL_n(\Z ) \rightarrow \GL _n (\Z / m\Z)$ 
is torsion free.
\end{lemma}

\begin{proof} 
Easy exercise ! One may deduce it from general results on formal groups over local rings, cf. Bourbaki [LIE III], \S7. Many variants exist. For instance:
\vskip 0.2cm
\noindent{\bf Lemma 1$^\prime$.} {\em Let $R$ be a local ring with maximal ideal ${\mathfrak m}$ and residue field $k = R/\mathfrak m.$ If $\ell$ is a prime number distinct from {\rm char}$(k)$, the kernel of the map $\GL_n(R) \rightarrow \GL_n(k)$ does not contain any element of order $\ell$. }
\vskip 0.2cm
\noindent{\sl Proof.} Suppose $x \in \GL_n(R)$ has order $\ell$ and gives 1 in  $\GL_n(k)$. Write $x  = 1 + y$; all the coefficients of the matrix $y$ belong to $\mathfrak m$. Since $x^\ell = 1$, we have 
$$
\ell \cdot y + {\ell\choose2}\cdot y^2 + \dots + \ell\cdot y^{\ell -1} + y^\ell = 0,
$$
which we may write as $y\cdot u = 0$, with $u = \ell + {\ell\choose2} y + \dots + y^{\ell -1}$. The image of $u$ in $\GL_n(k)$ is $\ell$, which is invertible. Hence $u$ is invertible, and since $y\cdot u$ is 0, this shows that $y = 0$.
\end{proof}

Several other variants can be found in [SZ 96].

\vskip 0.3cm
\noindent {\sl Remark.}
A nice consequence of lemma $1^\prime$ is the following result of Malcev and Selberg ([Bo 69], \S 17):\smallskip

$(^*)$ {\sl Let} $\Gamma$ {\sl be a finitely generated subgroup of} $\GL_n(K)$, {\sl where} $K$ {\sl is a field of characteristic} $0$. {\sl Then} $\Gamma$ 
{\sl has a torsion free subgroup of finite index}.\smallskip

\noindent {\sl Sketch of proof} (for more details, see Borel, {\sl loc.cit.}). Let $S$ be a finite generating subset of $\Gamma$, and let $L$ be the ring generated by the coefficients of the elements of $S \cup S^{-1}$. We have $\Gamma \subset \GL_n (L)$. Let $\mathfrak m$ be a maximal ideal of $L$; the residue field $k = A/\mathfrak m$ is finite ([AC V], p.68, cor.1 to th.3); let $p$ be its characteristic. The kernel $\Gamma_1$ of $\Gamma \rightarrow {\bf GL}_n(k)$ has finite index in $\Gamma$; by lemma $1^\prime$ (applied to the local ring $R = L_{\mathfrak m}$), $\Gamma_1$ does not have any torsion except possibly $p$-torsion. By choosing another maximal ideal of $L$, with a different residue characteristic, one gets a torsion free subgroup of finite index of $\Gamma_1$, and hence of $\Gamma$.\hfill$\Box$

\vskip 0.3cm
\noindent{\sl Remark.} When $K$ has characteristic $p > 0$ the same proof shows that $\Gamma$ has a subgroup of finite index which is ``$p^\prime$-torsion free", i.e. such that its elements of finite order have order a power of $p$.

\subsection{Proof of theorem 1 (i).} Let $A$ be a finite subgroup of $\GL_n(\Q)$; we have to show that $v_\ell(A) \leq M(n,\ell )$. Note first:

\subsubsection{The group $A$ is conjugate to a subgroup of $\GL_n(\Z)$.}${}$

This amounts to saying that there exists an $A$-stable lattice in $\Q^n$, 
which is clear: just take the lattice generated by the $A$-transforms of the 
standard lattice $\Z^n$.

\subsubsection{There is a positive definite quadratic form on $\Q^n$, with 
integral coefficients, which is invariant by $A$}  ${}$

Same argument: take the sum of the $A$-transforms of $x_1^2 + \dots + x_n^2$,
 and multiply it by a suitable non-zero integer, in order to cancel any denominator. \smallskip

\noindent Let us now proceed with the proof of $v_\ell (A) \leq M(n,\ell )$. We do it in two steps:

\subsubsection{The case $\ell > 2$} ${}$

By 1.3.1, we may assume that $A$ is contained in $\GL_n(\Z)$. Let $p$ be a prime number $\not= 2$. By lemma 1, the map $A \rightarrow \GL_n(\Z/p\Z)$ is injective. Hence
$$
v_\ell (A) \le a(p) = v_\ell \big(\GL_n(\Z/p\Z)\big).
$$
The order of $\GL_n(\Z/p\Z)$ is $p^{n(n-1)/2}(p-1)(p^2-1)\dots (p^n-1)$. Let us assume that $p \not= \ell$. Then we have
$$
a(p) = \sum^n_{i=1} v_\ell (p^{i}-1).
$$
We now choose $p$ in such a way that $a(p)$ is as small as possible. More precisely, we choose $p$ such that:

$(^*)$ {\emph{The image of  $p$ in $(\Z/\ell^2\Z)^*$ is a generator of that group.}}

This is possible by Dirichlet's theorem on the existence of primes in arithmetic progressions (of course, one should also observe that $(\Z/\ell^2\Z)^*$ is cyclic.)

Once $p$ is chosen in that way, then $p^{i}-1$ is divisible by $\ell$ only if $i$ is divisible by $\ell-1$; moreover, one has $v_\ell (p^{\ell-1}-1) = 1$ because of $(^*)$, and this implies that $v_\ell(p^{i}-1) = 1 + v_\ell (i)$ if $i$ is divisible by $\ell-1$. (This is where the hypothesis $\ell > 2$ is used.) One can then compute $a(p)$ by the formula above. The number of indices $i \le n$ which are divisible by $\ell-1$ is $\left[\frac{n}{\ell-1}\right]$. We thus get:
\begin{eqnarray*}
a(p) & = & \left[\frac{n}{\ell-1}\right] + \sum_{1 \le j \le \left[\frac{n}{\ell -1}\right]} v_\ell (j) = \left[\frac{n}{\ell-1}\right] + v_\ell \big(\left[\frac{n}{\ell-1}\right]!\big)\\
& = & \left[\frac{n}{\ell-1}\right] + \left[\frac{n}{\ell (\ell -1)}\right] + \dots = M(n,\ell ).
\end{eqnarray*}
\noindent This proves th.1 (i) in the case $\ell \not= 2$. 

\subsubsection{The case $\ell = 2$.}${}$

When $\ell = 2$, the method above does not give the right bound as soon as $n > 1$. One needs to replace $\GL_n$ by an orthogonal group. Indeed, by 1.3.1 and 1.3.2,  we may assume, not only that $A$ is contained in $\GL_n(\Z)$, but also that it is contained in the orthogonal group $\O_n(q)$, where $q$ is a non-degenerate quadratic form with integral coefficients. Let $D$ be the discriminant of $q$, and let us choose a prime number $p > 2$ which does not divide $D$. The image of $A$ in $\GL_n(\Z/p\Z)$ is contained in the orthogonal group $\O_n (\Z/p\Z)$ relative to the reduction of $q$ mod $p$. If we put $r = [n/2]$, the order of $\O_n(\Z/p\Z)$ is known to be:
$$2\cdot p^{r^2} (p^2-1) (p^4-1) \dots (p^{2r}-1)\quad {\mbox{ if $n$ is odd.}}
$$

\noindent and
$$2\cdot p^{r(r-1)}(p^2-1)(p^4-1)\dots (p^{2r}-1)/(p^r+\varepsilon)\quad {\mbox{if $n$ is even,}}
$$
with $\varepsilon = \pm 1$ equal to the Legendre symbol at $p$ of $(-1)^rD.$

If we choose $p \equiv \pm 3$ (mod 8), we have $v_2(p^{2i}-1) = 3 + v_2(i)$, and 
 $v_2(p^r+\varepsilon ) \ge 1$. If $n$ is odd, this gives 
 $$v_2(\O_n(\Z/p\Z)) = 1 + 3r + v_2(r!) = n + r + \left[\frac{r}{2}\right] + \left[\frac{r}{4}\right] + \dots = M(n,2),$$ and, if $n$ is even:
$$v_2(\O_n(\Z/p\Z)) \le 3r + v_2(r!) = M(n,2).$$
Hence $v_2(A)$ is at most equal to $M(n,2)$.\hfill $\Box$
\vskip 0.3cm
\noindent {\sl Remark.} There are several ways of writing down this proof. For instance:

- There is no need to embed $A$ in $\GL_n(\Z)$. It sits in $\GL_n(\Z[1/N])$ for a suitable $N \ge 1$, and this allows us to reduce mod $p$ for all $p$'s not dividing $N$. 

- Minkowski's lemma is not needed either: we could replace it by the trivial fact that a matrix which is different from 1 is not congruent to 1 \mbox{$\mod p$} for all large enough $p$'s.

- Even when $\ell > 2$, we could have worked in $\O_n$ instead of $\GL_n$; that is what Minkowski does. 

- When $\ell = 2$ the case $n$ even can be reduced to the case $n$ odd by observing that, if $A \subset \GL_n (\Q)$, then $A\times \{\pm 1\}$ embeds into $\GL_{n+1}(\Q)$, and $M(n+1,2)$ is equal  to $1 + M(n,2)$.

\subsection{Proof of theorem 1 (ii).} The symmetric group $S_\ell$ has a faithful representation $S_\ell \rightarrow \GL(V_1)$ where $V_1$ is a $\Q$-vector space of dimension $\ell-1$. Put $r = \left[\frac{n}{\ell-1}\right]$, and let $V = V_1 \oplus \dots \oplus V_r$ be the direct sum of $r$ copies of $V_1$. Let $S$ be the semi-direct product of $S_r$ with the product $(S_\ell )^r$ of $r$ copies of $S_\ell$ (``wreath product"). The group $S$ has a natural, and faithful, action on $V$. We may thus view $S$ as a subgroup of $\GL_{r(\ell -1)}(\Q)$, hence also of $\GL_n(\Q)$, since $n \ge r(\ell-1)$. We have 
$$ v_\ell (S) = r + v_\ell (r!) = \left[\frac{n}{\ell-1}\right] + \left[\frac{n}{\ell(\ell-1)}\right] + \dots = M(n,\ell).$$
An $\ell$-Sylow $A$ of $S$ satisfies the conditions of th.1 (ii).\hfill $\Box$

\vskip 0.2cm
\noindent {\sl Example.}
When $\ell = 2$ the group $S$ defined above is the ``hyper-octahedral group", i.e. the group of automorphisms of an $n$-cube (= the Weyl group of a root system
of type $B_n$); in ATLAS notation, it may be written as $2^n\cdot S_n$. 

%1.5 A conjugacy theorem.
\subsection{A conjugacy theorem.} The finite $\ell$-subgroups of $\GL_n(\Q)$ have the following Sylow-like property:

\begin{mytheorem1'} Let $A$ and $A^\prime$ be two finite $\ell$-subgroups of $\GL_n(\Q)$. Assume that $A$ has the maximal order allowed by th.$1$. Then $A^\prime$ is conjugate to a subgroup of $A$.
\end{mytheorem1'}

\begin{corollary}
%\label{cor1}
If $|A| = |A^\prime| = \ell^{M(n,\ell)}$, then $A$ and $A^\prime$ are conjugate in $\GL_n(\Q)$.
\end{corollary}

\vskip 0.5cm
\noindent{\sl Proof of theorem $1^\prime$.} See Bourbaki, [LIE III], \S7, exerc.6 f) where only the case $\ell > 2$ is given, and Feit [Fe 97] who does the case $\ell = 2$. Let us sketch Bourbaki's method (which we shall use in \S6.6 in a more general setting):

We may assume that $A$ and $A^\prime$ are contained in $\GL_n(\Z)$. Choose a prime $p$ as in 1.3.3, and reduce mod $p$. The groups $A$ and $A^\prime$ then become {\mbox{$\ell$-subgroups}} of $G_p = \GL_n(\Z/p\Z)$, and $A$ is an $\ell$-Sylow of $G_p$. By Sylow's theorem applied to $G_p$, one finds an injection $i : A^\prime \rightarrow A$ which is induced by an inner automorphism of $G_p$. The two linear representations of $A^\prime$:
$$
A^\prime \rightarrow \GL_n(\Q)\quad {\mbox {\rm and}}\quad A^\prime \stackrel{i}{\rightarrow} A \rightarrow \GL_n(\Q)$$
become isomorphic after reduction mod $p$. Since $p \not= \ell$, a standard argument shows that they are isomorphic over $\Q$, which proves th.$1^\prime$ in that case. The case $\ell = 2$ can be handled by a similar, but more complicated, argument: if $n$ is odd, one uses orthogonal groups as in 1.3.4, and one reduces the case $n$ even to the case $n$ odd by the trick mentioned at the end of \S 1.3. \hfill$\Box$
\vskip 0.3cm
\noindent{\small{\sl Exercise.} Let $A(n)$ be a maximal 2-subgroup of $\GL_n(\Q)$. Show that the $A(n)$'s can be characterized by the following three properties:
\begin{eqnarray*}
A(1) & = & \{\pm 1\}.\\
A(2n) & = & \big(A(n)\times A(n)\big)\cdot\{\pm 1\} \,\, {\mbox{(wreath product) if $n$ is a power of $2$.}}\\
A(n) & = & A(2^{^{m_1}})\times \dots \times A(2^{^{m_k}})\,{\mbox{if $n$ = $2^{^{m_1}} \!\!+\dots + 2^{^{m_k}} $ with $m_1 < \dots < m_k$}}.
\end{eqnarray*}}

\vskip 0.5cm
\begin{center}
{\bf {\S 2. Schur}}
\vskip 0.3cm
\end{center}
\setcounter{section}{2}
\setcounter{subsection}{0}

Ten years after [Mi 87], Frobenius founded the theory of characters of finite
 groups. It was then (and still is now) very tempting to use that theory to 
give a different proof of Minkowski's results. The first people to do so were 
Schur ([Sch 05]) and Burnside ([Bu 11], Note G). Schur's paper is especially 
interesting. He works first over $\Q$, as Minkowski did, and uses a very 
original argument in character theory, see \S 2.1 below. He then attacks the 
case of an arbitrary number field, where he gets a complete answer, see \S 2.2.

\subsection{Finite linear groups with rational trace.} What Schur proves in \S1 of [Sch 05] is:

\begin{mytheorem2} Let $A$ be a finite $\ell$-subgroup of $\GL_n(\C).$ Assume that the traces of the elements of $A$ lie in $\Q$. Then $v_\ell(A) \le M(n,\ell )$, where $M(n,\ell)$ is as in th.$1$.
\end{mytheorem2}

The condition on the traces is obviously satisfied if $A$ is contained in $\GL_n(\Q)$. Hence th.2 is a generalization of th.1. (As a matter of fact, it is a genuine generalization only when $\ell = 2$; indeed, when $\ell > 2$, it is known, cf. [Ro 58], that a finite $\ell$-subgroup of $\GL_n(\C)$ with rational trace is conjugate to a subgroup of $\GL_n(\Q)$.)

\begin{proof} We start from the following general fact, which is implicit in [Sch 05] (and is sometimes called ``Blichfeldt's lemma"):

\begin{proposition}
Let $G$ be a finite subgroup of $\GL_n(\C)$ and let $X$ be the subset of $\C$ made up of the elements $\Tr(g)$ for $g \in G, g\not= 1.$ Let \mbox{$N = \prod (n-x)$} be the product of the $n-x$, for $x \in X$. Then $N$ is a non-zero integer which is divisible by $|G|.$
\end{proposition}

(Hence the knowledge of the set $X$ gives a multiplicative bound for the order of $G$.) 
\vskip 0.3cm
\noindent{\sl Proof.} Let $m = |G|$, and let $z$ be a primitive $m$-th root of unity. The elements of $X$ are sums of powers of $z$; hence they belong to the ring of integers of the cyclotomic field $K = \Q(z)$. This already shows that $N$ is an algebraic integer. If $s$ is an element of $\Gal(K/\Q)$, one has $s(z) = z^a$ for some $a\in (\Z/m\Z)^*$. If $x = \Tr(g)$, with $g\in G$, then $s(x) = \Tr(g^a)$, hence $s(x)$ belongs to $X$. This shows that $X$ is stable under the action of $\Gal(K/\Q)$; hence $N$ is fixed by $\Gal(K/\Q)$; this proves that $N$ belongs to $\Z$.

The factors of $N$ are $\not= 0$. Indeed, $\Tr(g)$ is equal to the sum of $n$ complex numbers $z_i$ with $|z_i| = 1$, hence can be equal to $n$ only if all the $z_i$ are equal to $1$, which is impossible since $g\not= 1$. This shows that $N\not= 0$ (one could also prove that $N$ is positive, but we shall not need it).

It remains to see that $N$ is divisible by $|G|$. It is well-known that, if $\chi$ is a generalized character of $G$, the sum $\sum_{g\in G} \chi(g)$ is divisible by $|G|$. Let us apply this to the function $g \mapsto \chi(g) = \prod_{x\in X}\big(\Tr(g) -x\big)$, which is a $\Z$-linear combination of the characters $g \mapsto \Tr(g)^m, m \ge 0$. Since $\chi(g) = 0$ for $g \not= 1$ and $\chi(1) = N$, the sum of the $\chi(g)$ is equal to $N$. Hence $N$ is divisible by $|G|$.\end{proof}

The next lemma gives an information on the $\Tr(g)$'s:

\begin{lemma} Let $A$ be as in th.$2$. If $g\in A$, then $\Tr(g)$ may be written as $n - \ell y$ with $y \in\Z$ and $0 \le y \le n/(\ell -1)$.
\end{lemma}

\begin{proof} Each eigenvalue of $g$ is of order $\ell^\alpha$ for some $\alpha \ge 0$, and all the eigenvalues with the same $\alpha$ have the same multiplicity. By splitting $\C^n$ according to the $\alpha$'s, one is reduced to the following three cases:

\noindent (1) $g = 1$ and $n = 1$. Here $\Tr(g) = 1$ and we take $y = 0.$

\noindent (2) $g$ has order $\ell$ and $n = \ell -1$. Here $\Tr(g) = -1$, and $y = 1$.

\noindent (3) $g$ has order $\ell^\alpha$ with $\alpha > 1$ and $n = \ell^{\alpha -1} (\ell -1)$. Here $\Tr(g) = 0$ and $y = \ell^{\alpha -2} (\ell -1).$

In each case we have $0 \le y \le n/(\ell -1).$\end{proof}
\vskip 0.3cm
\noindent{\sl End of the proof of theorem 2.} We apply prop.1 to $G = A$. By lemma 2, each factor $n-x$ of $N$ can be written as $\ell y$ with $1 \le y \le d = [n/\ell -1)]$. This shows that $N$ divides the product $\ell^d\cdot d$! and we have
$$
v_\ell(N) < d + v_\ell (d!) = [n/(\ell -1)] + [n/\ell (\ell -1)] + \dots = M(n,\ell ).
$$
Since $|G|$ divides $N$, this proves th.2. \hfill $\Box$

\vskip 0.3cm
\noindent{\sl Remark.} One may ask whether th.2 can be complemented by a conjugacy theorem analogous to th.$1^\prime$ of \S 1.5. The answer is of course ``yes" if $\ell > 2$ (because of th.$1^\prime$), but it is ``no" for $\ell = 2$: the dihedral group $D_4$ and the quaternion group $Q_8$ are non-conjugate $2$-subgroups of $\GL_2(\C)$, with rational trace, which have the maximal order allowed by th.2, namely 8.

\subsection{Replacing $\Q$ by an arbitrary number field.} This is what Schur does in \S\S 2-6 of [Sch 05]. Before stating his result, some notation is necessary: 

- $k$ is a number field, viewed as a subfield of $\C$.

- For each $a \ge 1$, $z_a$ denotes a primitive $a$-th root of unity. 

- (assuming $\ell > 2$). We put $t = [k(z_\ell ):k]$ and we denote by $m$ the maximal $a$ such that $k(z_\ell )$ contains $z_{_{\ell^a}}$ (this notation coincides with Schur's, and it will be extended to arbitrary fields in \S 4 of Lect.II). We put
$$
M_k(n,\ell ) = m\cdot \left[\frac{n}{t}\right] + \left[\frac{n}{\ell t}\right] + \left[\frac{n}{\ell^2t}\right] + \dots
$$
- (assuming $\ell = 2$).  We put $t = [k(i):k]$ and we define $m$ as explained in \S 4.2 (warning: $t$ and $m$ do not always coincide with Schur's $t_2$ and $m_2$). We put:
$$
M_k(n,2) = n + (m^\prime -1) \left[\frac{n}{t}\right] + \left[\frac{n}{2t}\right] + \left[\frac{n}{4t}\right] + \dots ,
$$
where $m^\prime$ is equal to $m+1$ in case (b) of \S4.2 and is equal to $m$ in the other cases.

The main result of [Sch 05] is:

\begin{mytheorem2'} Let $A$ be a finite $\ell$-subgroup of $\GL_n(\C)$ such that $\Tr(g)$ belongs to $k$ for every $g\in A$. Then $v_\ell (A) \le M_k(n,\ell )$.
\end{mytheorem2'}

\noindent Note that, when $k = \Q$, the integer $M_k(n,\ell )$ is equal to Minkowski's $M(n,\ell )$; hence th.2$^\prime$ is a generalization of th.2.

\begin{proof} I shall not give all the details of Schur's proof, but just
explain its main steps. For more information, see [Sch 05] (and also [GL 06] for the case $\ell > 2$).

One of the inputs of the proof is the following result, which had just been proved by Blichfeldt ([Bl 04] - see also \S3 below):

\subsubsection{Every linear representation of $A$ is monomial}${}$

Hence one can decompose the vector space $\C^n$ as a direct sum of $n$ lines $D_1,\dots ,D_n$ which are permuted by $A$. This gives a homomorphism $A \rightarrow S_n$; its kernel $A^\prime$ is a normal abelian subgroup of $A$. Hence:

\subsubsection{The group $A$ has a normal abelian subgroup $A^\prime$ such that $(A:A^\prime)$ divides $n!$}${}$

This led Schur to investigate the case where $A$ is abelian. He proved:

\subsubsection{If $A$ is as in th.$2^\prime$, and is abelian, then\ {\rm :}}
$$
v_\ell (A) \le \left\{ \begin{array}{ll}
m\cdot \left[\frac{n}{t}\right] &{\mbox{\sl if}}\,\, \ell > 2\\
&\\
(m^\prime - t)\cdot \left[\frac{n}{t}\right] + n & {\mbox{\sl if}}\,\, \ell = 2.\end{array}
\right.
$$

\noindent{\sl Sketch of proof.} Since $A$ is abelian, and the traces of its 
elements belong to $k$, it is conjugate to a subgroup of $\GL_n(k)$. Let $R$ 
be the $k$-subalgebra of ${\bf M}_n(k)$ generated by $A$.   We may write $R$ as a 
product $\prod K_i$, where the $K_i$ are cyclotomic extensions of $k$, of the 
form $k(z_{_{\ell^{a_i}}})$, with $a_i \ge 0$. Let $n_i = [K_i:k]$; then
$\sum n_i \le n$. The image of $A$ in $K^*_i$ is a cyclic group of order 
$\ell^{^{a_i}}$. If $\ell > 2$, it is not difficult to see that 
$a_i \le m\cdot\!\left[\frac{n_i}{t}\right]$ for every $i$. Adding up, we find 
$\sum a_i \le m\cdot \left[\frac{n}{t}\right]$, and since 
$v_\ell (A) \le \sum a_i$, we get the inequality (2.2.3). The case $\ell = 2$ is similar. 
\end{proof}

Once this is done, the case $\ell = 2$ follows. Indeed (2.2.2) and (2.2.3) give $v_2(A) \le v_2(A^\prime) + v_2(n!) \le n + (m^\prime -t)\cdot \left[\frac{n}{t}\right] + v_2(n!)$, and this is equivalent to $v_2(A) \le M_k(n,2)$. The case $\ell > 2$ requires more work, cf. [Sch 05], \S 5.\hfill $\Box$

\vskip 0.3cm
\noindent{\sl Remarks}

1) The bound $v_\ell(A) \le M_k(n,\ell )$ is {\sl optimal}; this is proved by the same explicit constructions as in \S 1.4, cf. [Sch 05], \S 6.

2) As we already pointed out in \S 2.1, the hypothesis $\Tr(A) \subset k$ implies, when $\ell > 2$, that $A$ is conjugate to a subgroup of $\GL_n(k)$. One may then use Minkowski's method, as will be explained in \S 6 for semisimple algebraic groups (of course $\GL_n$ is not semisimple, but the method applies with almost no change -- the invariant degrees $d_i$ of \S 6 have to be replaced by $1,2,\dots ,n)$. The bound found in that way coincides with Schur's. 

For $\ell = 2$, if one does not assume that $A$ can be embedded in $\GL_n(k)$, I do not see how to apply either Minkowski's method or the cohomological method of \S 6.8. This raises interesting questions. For instance, consider a finite subgroup $A$ of $E_8(\C)$, and suppose that the conjugacy classes of the elements of $A$ are $\Q$-rational. Is it true that $v_2(A) \le 30$, $v_3(A) \le 13, \dots ,$ as would be the case if $A$ were contained in the rational points of a $\Q$-form of $E_8$, cf. \S 6.3.2 ?

\vskip 0.5cm
\begin{center}
{\bf {\S 3. Blichfeldt and others}}
\vskip 0.3cm
\end{center}
\setcounter{section}{3}
\setcounter{subsection}{0}

Blichfeldt's theorem (\S 3.1 below) has already been used in \S 2.2. The results of \S 3.3 will be applied in \S 5.4, in order to prove what I call the ``S-bound". 

\subsection{Blichfeldt's theorem.} Recall that a finite group $A$ is called {\sl supersolvable} if it has a composition series
$$
1 = A_0 \subset A_1 \subset \dots \subset A_m = A
$$
where the $A_i$ are normal in $A$ (and not merely in $A_{i+1})$ and the quotients $A_i/A_{i-1}$ are cyclic. One has
\begin{center}
nilpotent $\Rightarrow$ supersolvable $\Rightarrow$ solvable.
\end{center}
In particular, an $\ell$-group is supersolvable.

One proves easily:

$(^*)$ {\sl If $A$ is supersolvable and non abelian, there exists an abelian normal subgroup $A^\prime$ of $A$ which is not contained in the center of $A$.}

Recall also that a linear representation $V$ of a group $A$ is called {\sl monomial} if one can split $V$ as a direct sum of lines which are permuted by $A$. When $V$ is irreducible, this amounts to saying that $V$ is induced by a 1-dimensional representation of a subgroup of $A$.

We may now state Blichfeldt's theorem ([Bl 04], see also [Bu 11], \S 258):

\begin{mytheorem3} Every complex linear representation of a supersolvable finite group is monomial.
\end{mytheorem3}

\noindent (As a matter of fact, Blichfeldt was only interested in the case where $A$ is nilpotent.)

\begin{proof}
The argument is now standard. We may assume that the given representation $V$ is irrreducible and faithful. If $A$ is abelian, we have $\dim V = 1$ and there is nothing to prove. If not, we choose 
$A^\prime$ as in ${(^*)}$ above, and we split $V$ as $V = \oplus V_\chi$, where $\chi$ runs through the 1-dimensional characters of $A^\prime$, and $V_\chi$ is the corresponding eigenspace; let $V_\psi$ be a non-zero $V_\chi$; it is distinct from $V$ (otherwise, $A^\prime$ would be central), and every non-zero $V_\chi$ is an $A$-transform of $V_\psi$ (because $V$ is irreducible). Call $B$ the subgroup of $A$ stabilizing $V_\psi$. We have $A^\prime \subset B \subset A$, and $|B| < |A|$. Using induction on $|A|$, we may assume that th.3 is true for $B$; this gives a splitting of $V_\psi$ as a direct sum of lines which are stable under $B$. By transforming them by $A$, we get the desired splitting of $V$.
\hfill$\Box$

\subsection{Borel-Serre.} In [BS 53], Borel and I proved:

\begin{mytheorem3'}
Let $G$ be a compact real Lie group, and let $A$ be a finite supersolvable subgroup of $G$. There exists a maximal torus $T$ of $G$ which is normalized by $A$. 
\end{mytheorem3'}

\vskip 0.3cm
\noindent{\sl Remark.} When one applies th.3$^\prime$ to $G = \U_n(\C)$, one recovers th.3. Hence th.3$^\prime$ may be viewed as a generalization of Blichfeldt's theorem. 

\vskip 0.3cm
\noindent{\sl Proof of theorem 3 \!$^\prime$} ({\sl sketch}).

\begin{lemma}
Let $\mathfrak{g}$ be a finite dimensional Lie algebra over a field of characteristic 0, and let $s$ be an automorphism of prime order of $\mathfrak{g}$. If $s$ has no fixed point $\not= 0$, then $\mathfrak{g}$ is nilpotent.
\end{lemma}

\noindent (Note the analogy with a - much deeper - theorem of Thompson \linebreak[4][Th\nolinebreak[4] 60-64]: if a finite group $G$ has an automorphism of prime order with no non-trivial fixed point, then $G$ is nilpotent.)

\vskip 0.3cm
\noindent{\sl Proof of lemma 3.} By extending scalars, we may assume that the ground field is algebraically closed. Let $p$ be the order of $s$, and let $z$ be a primitive $p$-th root of unity. Let $\mathfrak{g}_i$ be the kernel of $s - z^i$ in $\mathfrak{g}$. We have
$$
\mathfrak{g} = \mathfrak{g}_0 \oplus \mathfrak{g}_1 \oplus \dots \oplus \mathfrak{g}_{p-1}\, ,
$$
 and the hypothesis made on $s$ means that $\mathfrak{g}_0 = 0$. One then shows that $\ad(x)^{p-1} = 0$ for every $x$ belonging to one of the $\mathfrak{g}_i$'s. This implies that the Killing form of $\mathfrak{g}$ is 0, hence that $\mathfrak{g}$ is solvable (Cartan's criterion). The fact that $\mathfrak{g}$ is nilpotent follows easily. (For more details, see \S 4 of [BS 54].)\hfill$\Box$
\vskip0.2cm 
 Once this is done, th.3 \!\!$^\prime$ is proved by an induction argument similar to the one used in the proof of Blichfeldt's theorem, cf. [BS 53], \S 3.\end{proof}

\subsection{Steinberg and Springer-Steinberg.} We now come to the setting of
 linear algebraic groups. Let $k$ be a field, and let $G$ be an algebraic
 group over $k$. We shall assume in what follows that $G$ is linear and 
smooth over $k$; the connected component of the identity of $G$ is denoted by 
$G^\circ$. Recall that $G$ is said to be {\sl reductive} if it is connected 
and if its unipotent radical (over an algebraic closure of $k$) is trivial, cf. [Bo 91], \S 11.21. If $k = \C$,
 such groups correspond (by a standard dictionary, cf. [Se 93], \S 5) to the 
connected compact Lie groups. [In the literature, a group $G$ such that 
$G^\circ$ is reductive is sometimes called ``reductive"; this is reasonable in
 characteristic 0, but not otherwise. Here we prefer that ``reductive" implies ``connected".]

Theorem 3 \!\!$^\prime$ has the following analogue:

\begin{mytheorem3''} Let $A$ be a finite supersolvable group of order prime to $\car(k)$ and let $G$ be a reductive group over $k$ on which $A$ acts by $k$-automorphisms. Then there exists a maximal torus $T$ of $G$, defined over $k$, which is stable under the action of $A$.
\end{mytheorem3''}

\noindent (When $k = \C$, this is equivalent to th.3 \!$^\prime$, thanks to the dictionary mentioned above.)

\begin{corollary}
If $A$ is a finite supersolvable subgroup of $G(k)$, of order prime to $\car (k)$, there is a maximal $k$-torus $T$ of $G$ whose normalizer $N$ is such that $A$ is contained in $N(k)$.
\end{corollary}

\noindent (Recall that, if $X$ is a $k$-variety, $X(k)$ is the set of $k$-points of $X$.)

\vskip 0.3cm
\noindent{\sl Proof of theorem 3 \!$^{\prime\prime}$.} When $k$ is algebraically closed, this is proved in [SS 68], I.5.16, with the help of several results from [St 68]. For an arbitrary field $k$, the same proof works with very little change. One starts with the following basic result of Steinberg ([St 68], th.7.2):

\begin{proposition}
 Assume $k$ is algebraically closed. Let $s : G \rightarrow G$ be a surjective homomorphism. Then there exists a Borel subgroup $B$ of $G$ such that $s(B) = B$.
 \end{proposition}
 
 When $s$ has finite order prime to $\car(k)$, one can say much more:

\begin{proposition}
Let $s$ be an automorphism of $G$ of finite order prime to $\car(k)$, and let $G^s$ be the subgroup of $G$ fixed by $s$. Then\hbox{ \!{\rm :}} 

{\rm a)} The connected component of $G^s$ is reductive.

{\rm b)} One has $\dim G^s > 0$ if $G$ is not a torus.

{\rm c)} If $k$ is algebraically closed, there exists a Borel subgroup $B$ of $G$ and a maximal torus $T$ of $B$ such that $s(B) = B$ and $s(T) = T$.
\end{proposition}

\vskip 0.3cm\noindent{\sl Proof} ({\sl sketch}). We may assume $k$ is algebraically closed, since assertions a) and b) are ``geometric". A proof of a) is given in [St 68], cor.9.4. A proof of c) is given in [SS 68], I.2.9, as an application of prop.2. Assertion b) follows from c) by the following method of Steinberg: one observes that a pair $(B,T)$ with $B\supset T$, determines {\sl canonically} a homomorphism $h : \G_m \rightarrow T$ (indeed $B$ gives a basis of the root system of $(G,T)$, and one takes for $h$ twice the sum of the corresponding coroots). Moreover, $h$ is non-trivial if $G$ is not a torus. The canonicity of $h$ implies that it is fixed by $s$. Hence $G^s$ contains $\Im(h)$.\hfill$\Box$
\vskip 0.3cm
\noindent{\sl End of the proof of th.3 \!$^{\prime\prime}$}. By induction on 
$|A| + \dim G$. When $A = 1$, one takes for $T$ any maximal $k$-torus of 
$G$; by a theorem of Grothendieck, there is such a torus 
(cf. [Bo 91], th.18.2). We may thus assume $A\not= 1$. In that case $A$ 
contains a cyclic subgroup $<s>$, non-trivial, which is normal. We may also 
assume that $G$ is semisimple and that $A$ acts faithfully. Let $G_1$ be the 
connected component of $G^s$; we have $\dim G_1>0$, cf. prop.3 b). The group 
$A/A^\prime$ acts on $G_1$. By the 
induction assumption, there is a maximal torus $T_1$ of $G_1$, defined over 
$k$, which is stable under the action of $A/A^\prime$, hence of $A$. 
Let $G_2$ be the centralizer of $T_1$ 
in $G$. It is a reductive group of the same rank as $G$. We have 
$\dim G_2 < \dim G$, since $T_1$ is not contained in the center of $G$. 
Moreover, $G_2$ is stable under the action by $A$. By applying the 
induction assumption to the pair $(G_2,A)$ we get a maximal $k$-torus $T$ 
of $G_2$ which is $A$-stable. Since $G_2$ and $G$ have the same rank, $T$ 
is a maximal torus of $G$. \hfill$\Box$

%%%%%%%%%%%%%%%%%%%%%%%%%%%%%%%%%%%%%%%%%%%%%%%%%%%%%%%%%%%%%%
%%%%%%%%%%%%%%%%%%%%%%%%%%%%%%%%%%%%%%%%%%%%%%%%%%%%%%%%%%%%%%
%\include{serredonna_2}
%\markboth{J.-P. SERRE, ORDERS OF FINITE SUBGROUPS}{bidon}
 
\markright{LECTURE II: UPPER BOUNDS}
%\markleft{J.-P. SERRE, ORDERS OF FINITE SUBGROUPS}

%\begin{center}
%\bf{\Large
%Lecture II

%Lower bounds: The Minkowski method}
%\end{center}

%\part*{Lecture II \\  Upper bounds}

\specialsection*{\bf  II. Upper bounds}

\setcounter{section}{4}
\setcounter{subsection}{-1}
\setcounter{theorem}{3}
Let $G$ be a reductive group over a field $k$, and let $\ell$ be a prime number, different from $\car(k)$. Let $A$ be a finite subgroup of $G(k)$. We want to give an upper bound for $v_\ell (A)$, in terms of invariants of $G$, $k$ and $\ell$. We give two such bounds. The first one (\S 5) is less precise, but very easy to apply; we call it the  S-bound (S for Schur). The other bound (\S 6) is the M-bound (M for Minkowski). Both bounds involve some cyclotomic invariants of $k$, which are defined in \S 4 below.
\vskip 0.5cm
\begin{center}
{\bf {\S 4. The invariants $t$ and $m$}}
\end{center}
%\section*{\S 4. The invariants $t$ and  $m$}
\subsection{Cyclotomic characters}
\label{notation}
Let $\bar{k}$ be an algebraic  closure of $k$, and let $k_s$ be the separable closure of $k$ in $\bar{k}$. For each $n \ge 1$ prime to $\car(k)$, let $\mu_n \subset k^\ast_s$ be the group of $n$-th roots of unity and let $z_n$ be a generator of $\mu_n$.  

The Galois group $\Gamma_k = \Gal(k_s/k)$ acts on $\<z_n\> = \mu_n$. This action defines a continuous homomorphism
$$\chi_{_n}:\Gamma_k \rightarrow \Aut(\mu_n) = (\Z/n\Z)^*,$$ 
which is called the $n$-{\sl th cyclotomic character of} $k$. 

This applies in particular to  $n = \ell^d \, (d = 0,1,\dots )$; by taking inverse limits we get the $\ell^\infty$-cyclotomic character
$$\chi_{_{\ell^\infty}}:\Gamma_k  \rightarrow \Z^*_\ell = \ilim (\Z /\ell^d\Z)^*,$$ 
where $\Z_\ell$ is the ring of $\ell$-adic integers. What matters for us is the image $\Im$ $\chi_{_{\ell^\infty}}$, which is a closed subgroup of $\Z^*_\ell$. To discuss its structure, it is convenient to separate the cases $\ell \not= 2$ and $\ell = 2$.

\subsection{The case $\ell \not= 2$}
We have
$$\Z^*_\ell = C_{\ell -1} \times \left\{ 1 + \ell\!\cdot\!\Z_\ell\right\}$$
where $C_{\ell -1}$ is cyclic of order $\ell - 1$ (i.e. $C_{\ell -1}$ is the group $\mu_{\ell -1}$ of the \linebreak$\ell$-adic field $\Q_\ell$; it is canonically isomorphic to $\F^*_\ell$). As for $1 + \ell\!\cdot \!\Z_\ell$, it is procyclic, generated by $1 + \ell$, and isomorphic to the additive group $\Z_\ell$; its closed subgroups are the groups $1 + \ell^d\!\cdot\! \Z_\ell = \, \<1 + \ell^d\>$, $d = 1,2,\dots , \infty$, with the convention $\ell^\infty = 0$. 

Since $\ell - 1$ and $\ell$ are relatively prime, the subgroup $\Im \chi_{_{\ell^\infty}}$ of $\Z^*_\ell$ decomposes as a direct product:
$$\Im \chi_{_{\ell^\infty}} = C_t \times \left\{ 1 + \ell^m \!\cdot\! \Z_\ell\right\}\, ,$$
where $t$ is a divisor of $\ell-1$, $C_t$ is cyclic of order $t$ and $m = 1,2, \dots$ or $\infty$. 
\vskip 0.1cm
\noindent {\sl Remark.}
An alternative definition of the invariants $t$ and $m$ is:
\begin{eqnarray*}
t & = & \left[ k(z_\ell ) : k\right] = k{\mbox{-degree of}}\,\,  z_\ell\\
m & = & {\mbox{upper bound of the}} \,\, d\ge 1\,\, {\mbox{such that}} \,\, z_{_{\ell^d}} \,\, {\mbox{is contained in}} \,\, k(z_\ell ).
\end{eqnarray*}
\noindent {\sl Examples.} If $k = \Q$ or $\Q_\ell$, $\chi_{_{\ell^\infty}}$ is surjective and we have $t = \ell -1$, $m = 1$. If $k = k_s$, then $\chi_{_{\ell^\infty}}$ is trivial and $t = 1$, $m = \infty$. If $k$ is finite with $q$ elements, $\Im \chi_{_{\ell^\infty}}$ is the closed subgroup of $\Z^*_\ell$ generated by $q$ and we have: 

\begin{eqnarray*}
t & = &{\mbox{order of}} \,\,\, q\,\,\, {\mbox{in}} \,\,\, \F^*_\ell\,\quad\quad\quad\quad\quad\quad\quad\quad\quad\quad\quad\quad\quad\quad\quad\quad\quad\quad\quad\\
m & = & v_\ell (q^t - 1) = v_\ell\,\, (q^{\ell -1} - 1)\, .\quad\quad\quad\quad\quad\quad\hfill
\end{eqnarray*}
\subsection{The case $\ell = 2$} Here $\Z^*_2 = C_2 \times \{ 1 + 4\!\cdot \!\Z_2\}$, where $C_2 = \left\{ 1, -1\right\}$ and the multiplicative group $1 + 4
\!\cdot \!\Z_2$ is isomorphic to the additive group $\Z_2$. There are three possibilities for $\Im \chi_{_{2^\infty}}$:

\begin{enumerate}
\item[(a)] $\Im \chi_{_{2^\infty}} = 1+2^m \! \cdot\!\Z_2 = \<1 + 2^m\>$, with $m = 2, \dots ,\infty$. 
We put $t=1$.
\item[(b)]  $\Im \chi_{_{2^\infty}} = \<-1+2^m \> $, with $m = 2,\dots ,\infty$. 
We put $t=2$.
\item[(c)]  $\Im \chi_{_{2^\infty}} = C_2 \times \{1 + 2^m\! \cdot\!\Z_2\} = \<-1,1+2^m \>$, $m = 2,\dots ,\infty$. 
We put $t=2$.
\end{enumerate}

If $m < \infty$, these types are distinct. If $m = \infty$, types (b) and (c) coincide; in that case $\Im \chi_{_{2^\infty}}$ is equal to $C_2$.

\vskip 0.5cm
\noindent{\sl Remark.} We have $t = [k (i) : k]$ with the usual notation $i = z_4$. Hence case (a) means that $-1$ is a square in $k$,  and  in  that case $m$ is the largest $d \ge 2$ such that $z_{_{2^d}} \in k$. 

If $t = 2$, case (c) is characterized by the fact that $-1$ belongs 
to $\Im \chi_{_{2^\infty}}$. As for $m$, it is given by: 
\begin{eqnarray*}
m & = & -1 + \mbox{ upper bound of the } d\geq2 \mbox{ such that } z_{_{2^d}} \in k(i) \mbox{ in case (b)}\hfill\\
m  & = & \mbox{upper bound of the } d\geq2 \mbox{ such that } z_{_{2^d}} \in k(i) \mbox{ in case (c)}.\hfill
\end{eqnarray*}

\noindent {\sl Examples.} If $k = \Q$ or $\Q_2$, we have type (c) with $t = 2, m = 2$. If $k = \R$, we have types (b) and (c) with $m = \infty$.  If $k$ is separably closed, we have type (a) with $ t = 1$ and $m = \infty$. 

When $\car (k) \not= 0$, type (c) is impossible unless $m = \infty$. If $k$ is finite with $q$ elements, we have type (a) with $m = v_2 (q-1)$ if $q \equiv 1$ (mod 4) and type (b) with $m = v_2 (q+1)$ if $q \equiv -1$ (mod 4).
\vskip 0.5cm
\subsection{The case of finitely generated fields.} Let $k_0$ be the prime subfield of $k$, i.e. $\Q$ if $\car(k) = 0$ or $\F_p$ if $\car(k) = p > 0$. Suppose that $k$ is {\sl finitely generated over} $k_0$. Then {\sl the invariant} $m$ {\sl is finite}, i.e. $\Im \chi_{_{\ell^\infty}}$ is infinite. 

Indeed, if not, there would be a finite extension $k'$ of $k$ containing the group $\mu$ of all the $\ell^d$-th roots of unity $(d = 1,2, \dots ).$ Let $K = k_0 (\mu )$ be the extension of $k_0$ generated by $\mu$. Then:
\begin{enumerate}
\item[(a)] $K$ is algebraic over $k_0$
\item[(b)] $K$ is finitely generated over $k_0$ (because it is contained in $k'$, cf. \newline [A V], \S14, cor. 3 to prop. 17).
\end{enumerate}

Hence $K$ is either a finite field or a number field, which is absurd since such a field only contains finitely many roots of unity.
\vskip 1cm
\begin{center}
%Paragraph 5
{\bf {\S 5. The S-bound}}
\end{center}
\vskip 0.5cm
% Dienstag, 21.2.06
\setcounter{section}{5}
\setcounter{subsection}{0}
We start with the case of tori:
\subsection{The S-bound for a torus: statements}
\begin{mytheorem4}
\label{thm4}
Let $T$ be a torus over $k$, and let $A$ be a finite subgroup of $T(k)$. Then
$$
v_\ell (A) \,\le\, m \left[ \frac{\dim T}{\varphi (t)}\right]\, ,
$$
where $m$ and $t$ are defined as in \S {\rm 4} above and $ \varphi$ is Euler's totient function. 
\end{mytheorem4}

\noindent The bound given by th.4 is optimal. More precisely:
\begin{mytheorem4'}
\label{thm4'}
Assume $m < \infty$. For every $n \ge 1$ there exist a $k$-{\sl torus} $T$ of dimension $n$ and a finite subgroup $A$ of $T(k)$ such that $v_{\ell} (A) = m \cdot [n/\varphi (t)].$
\end{mytheorem4'}
\vskip 0.4cm
\noindent {\sl Example.} Take $k = \Q$ and $\ell = 2$, so that $t = m = 2$. Then th.4 says that any finite 2-subgroup of $T(\Q)$ has order $\le 4^{\dim\, T}$, and th.4$^\prime$ says that this bound can be attained.
\vskip 0.5cm
\subsection{Proof of theorem 4.}
\setcounter{lemma}{3}
\begin{lemma}
\label{lem4}
Let $u \in \mathbf{M}_n(\Z_\ell)$ be an $n \times n$ matrix with coefficients in $\Z_\ell$, 
which we view as an endomorphism of $(\Q_\ell/\Z_\ell)^n$.
Then 
$$
v_\ell \big ( \ker(u) \big )  = v_\ell \big ( \det(u) \big ).
$$
\end{lemma}

\begin{proof}
This is clear if $u$ is a diagonal matrix, and one reduces the general case
to the diagonal one by multiplying $u$ on the right and on the left by invertible matrices.
\end{proof}

Now let $n$ be the dimension of the torus $T$. 
Let $Y(T) = \Hom_{k_s}(\G_m,T)$ be the group of cocharacters of $T$.
The action of $\Gamma_k$ on $Y(T)$ gives a homomorphism
$\rho:\Gamma_k \rightarrow \Aut \big ( Y(T) \big ) \cong \GL_n(\Z)$.
If we identify $T$ with $\G_m \times \dots \times \G_m$ (over $k_s$) by choosing a basis of $Y(T)$,
the $\ell^\infty$-division points of $T(k_s)$ form a group isomorphic to $(\Q_\ell/\Z_\ell)^n$
and the action of $g \in \Gamma_k$ on that group is by
$\rho(g) \chi(g)$, where $\chi = \chi_{_{\ell^\infty}}$.

\begin{lemma}
\label{lem5}
{\sl Let} $A$ {\sl be a finite subgroup of} $T(k)$. {\sl For every} $g \in \Gamma_k$ {\sl we have}
$$
v_\ell (A) \,\, \le\,\, v_\ell \big(\det (\rho (g)\, \chi (g) - 1)\big) = v_\ell \big(\det (\rho (g^{-1}) - \chi (g)\big).$$
\end{lemma}

\begin{proof}
 By replacing $A$ by its $\ell$-Sylow subgroup, we may assume that $A$ is an $\ell$-group, hence is contained in the $\ell$-division points of $T(k_s)$. Since the points of $A$ are rational over $k$, they are fixed by $g$, i.e. they belong to the kernel of $g\!-\!1$. The inequality then follows from lemma 4, applied to $u = \rho (g) \,\chi (g) - 1.$
\end{proof}

We now choose $g \in \Gamma_k$ such that the inequality of lemma 5
gives that of th.4.
Here is the choice:
$$
\quad  \chi(g) = z_tu,  \quad \mbox{ where } z_t\in \Z_{\ell}^* \mbox{ has order } t, \mbox{ and } v_\ell (1\!-\!u) = m.
$$
\big(This works for $\ell=2$ as well as for $\ell \neq 2$, thanks to the definition of $t$ in \S4.1 and \S4.2.
Note that in all cases but $\ell=2$, type (c), $\chi(g)$ is a topological generator of $\Im\chi$.\big)

We have $\rho(g) \in \GL_n(\Z)$, and $\rho(g)$ is of finite order
(because the image of $\rho:\Gamma_k \rightarrow \GL_n(\Z)$ is finite). 
Hence the characteristic polynomial $F$ of $\rho(g^{-1})$ is  a product of cyclotomic polynomials:

\numberwithin{equation}{section}
\numberwithin{equation}{subsection}
\setcounter{equation}{0}
\setcounter{subsection}{2}
\setcounter{subsubsection}{0}

\begin{equation}
\label{eq521}
F = \prod \Phi_{d_j}, \quad \mbox{ with } \sum \varphi(d_j) = n.
\end{equation}
The inequality of lemma~\ref{lem5} gives
%\numberwithin{equation}{subsection}
$$
\label{eq10}
v_\ell (A) \leq \sum v_\ell \big(\Phi_{d_j} (z_tu) \big).
$$
We thus need to compute $v_\ell \big(\Phi_d (z_tu)\big)$ for every $d \ge 1$. The result is:

\begin{lemma}
\label{lem6}
We have

$$
v_\ell \big( \Phi_d (z_tu ) \big) = \left\{
\begin{array}{ll}
m &{\mbox{if }} \,d=t \\
1 & {\mbox{if }} \,d = t\cdot \ell^\alpha, \ \alpha \geq 1 \ \mbox{ or } 
\  \alpha = -1  \ (\mbox{if } t=2=\ell)\\
0 & {\mbox{otherwise.}}
\end{array}
\right.
$$
\end{lemma}

\begin{proof}
(We restrict ourselves to the case $\ell \neq 2$.
The case $\ell=2$ is analogous but slightly different.)

We have $\Phi_d (z_tu) = \prod (z_tu-z)$ where $z$ runs through the primitive $d$-th roots of unity in $\overline{\Q}_\ell$. Write $d$ as $d = \delta\!\cdot\!\ell^\alpha$ with $(\delta , \ell) = 1$ and $ \alpha \ge 0$. The images of the $z$'s in the  residue field $\overline{\F}_\ell$ of $\overline{\Q}_\ell$ are primitive $\delta$-th roots of unity. If $\delta \not= t$, none of them is equal to the image of $z_tu$, which has order $t$. In that case, all the $z_tu -z$ are units in $\overline{\Q}_\ell$ hence have valuation 0 and we have $v_\ell \big(\Phi_d (z_tu)\big) = 0.$ If $\delta = t$, i.e. $d = t\!\cdot\!\ell^\alpha$ with $\alpha \ge 0$, there are two cases: 

(a) $\alpha = 0$, i.e. $d = t$. In that case, one of the $z$'s is equal to $z_t$ and we have $v_\ell (z_tu -z) = v_\ell (u-1) =m$; the other $z$'s contribute $0$. 

(b) $\alpha \geq 1$. Here $z$ can be written as $z'\!\cdot\!z''$ 
where $z'$ runs through the \mbox{$t$-th} primitive roots of $1$, 
and $z''$ through the $\ell^\alpha$-th primitive roots of $1$.
The valuation of $z-z_t u$ is $0$ unless $z' = z_t$, 
in which case $v_\ell(z-z_t u) = v_\ell(z'' - u)$.
It is well-known that $v_\ell (z''-1) = \frac{1}{(\ell-1)\ell^{\alpha-1}}$.
Since $v_\ell(u-1) = m$, which is strictly larger,  we have

$$v_\ell(z''-u) = v_\ell\big ( (z''-1) - (u-1) \big ) 
=  \frac{1}{(\ell-1)\ell^{\alpha-1}}= \frac{1}{\varphi(\ell^\alpha)}.$$

Since the number of the $z''$ is $\varphi(\ell^\alpha)$, 
we thus get $v_\ell\big ( \Phi_d(z_{t}u) \big ) =1$, as claimed.
\end{proof}

We can now prove theorem 4:
With the notation of~(\ref{eq521}), denote by $r_1$ the number of $j$'s with  $d_j=t$,
and by $r_2$ the number of $j$'s with $d_j = t\!\cdot\!\ell^{\alpha_j}$, $\alpha_j \geq 1$, 
or $\alpha_j=-1$ in case $\ell=2,t=2$.
Using lemmas 5 and 6 we get 
$$
v_\ell(A) \leq r_1m +r_2
$$
\mbox{ and of course }
$$
r_1 \varphi(t) + \sum \varphi(t\!\cdot\!\ell^{\alpha_j}) \leq n = \dim T.
$$
Since $ \varphi(t\!\cdot\!\ell^{\alpha_j}) \geq \varphi(t)(\ell-1)$ this shows that
$r_1 \varphi(t) + r_2 \varphi(t) (\ell-1)  \leq n$.

\noindent Hence $r_1 + r_2 (\ell -1) \le [n/\varphi (t)]$, and we have:
$$
v_{\ell} (A) \le r_1m + r_2\, \le \,r_1m + r_2 (\ell -1)m\, \le\,m [n/\varphi (t)]\, ,$$
which concludes the proof.
\hfill$\Box$
\vskip 0.2cm
\noindent {\sl Remark.}
Since $(\ell -1)m >0$ in all cases (even if $\ell = 2)$, the above proof shows that $v_\ell (A)$ can be equal to $m[n/\varphi (t)]$ only when $r_2 = 0$. In other words:

\vskip 0.2cm
\noindent{\bf Complement to theorem 4.} {\sl Assume} $v_\ell(A) = m[n/\varphi (k)]$, {\sl where}\linebreak
$n = \dim T$. {\sl If} $g \in \Gamma_k$ {\sl is such that} $\chi (g) = z_{t}u$, {\sl with} $v_\ell (u-1) = m$ {\sl as above, the characteristic polynomial of} $\rho (g)$ {\sl is divisible by} $(\Phi_t)^N$, {\sl with} $N = [n/\varphi (k)]$.

\noindent(In other words, the primitive  $t$-th roots of unity are eigenvalues of $\rho (g)$ with multiplicity $N$.)
\vskip 0.2cm
When $t = 1$ or 2 (i.e. when $\varphi (t) = 1$), this can be used to determine the structure of an ``optimal'' $T$:

\begin{corollary} Assume $t = 1$ or $2$, and $v_\ell (A) = mn.$ Then{\rm { :}}

{\rm (i)}  \,\,If $t  = 1$,  the torus $T$ is split {\rm (i.e. isomorphic to the
product of $n$ copies of }$\G_m${\rm )}.

{\rm (ii)} If $t = 2$, $T$ is isomorphic to the product of $n$ non-split  tori of dimension $1$ which are split by the quadratic extension $k(z_\ell )/k$ if $\ell \not= 2$ and by $k(i)/k$ if $\ell = 2$.

\end{corollary}

\begin{proof}
We give the proof for $t = 2$ and $\ell > 2$: the case $t =1$ is easier and the case $t = 2 = \ell$ requires similar, but more detailed, arguments. 

Let $\gamma \in \Gamma_k$. We may write $\chi (\gamma )$ as $e_\gamma\!\cdot\!u_\gamma$, with $e_\gamma \in \{1,-1\}$ and 
\mbox{$u_\gamma \in \{1 + \ell^m\Z_\ell\}$}. There are three cases:
\begin{enumerate}
\item[(a)] $e_\gamma = -1$ and $v_\ell (u_\gamma -1) = m$
\item[(b)] $e_\gamma = -1$ and $v_\ell (u_\gamma - 1) > m$
\item[(c)] $e_\gamma = 1.$
\end{enumerate}
In case (a), the ``complement" above shows that $\rho (\gamma )$ has $-1$ for eigenvalue with multiplicity $n$, hence $\rho (\gamma ) = -1$ in $\Aut (T) \simeq \G\L_n (\Z)$. 

In case (b), choose $g \in \Gamma_k$ of type (a); this is possible by the very definition of $t$ and $m$. The element $g^2\gamma$ is of type (a) (this uses the fact that $\ell$ is odd); hence we have $\rho (g^2\gamma) = -1$ and since $\rho (g) = -1$ this shows that $\rho (\gamma ) = -1$.

If $\gamma$ is of type (c), then $g\gamma$ is of type (a) or (b) and we have $\rho (g\gamma) = -1$ hence $\rho(\gamma) = 1.$

In all cases, we have $\rho (\gamma ) \in \{1,-1\}$, and more precisely $\rho (\gamma ) = e_\gamma$.
The corollary follows.\end{proof}
It would be interesting to have a similar classification for $t > 2$.

\subsection{Proof of theorem 4$^\prime$: construction of tori with large $A$'s} 
To prove th.4$^\prime$ it is enough to construct a $k$-torus $T$, of dimension $n = \varphi (t)$, such that $T(k)$ contains a cyclic subgroup of order $\ell^m$. Here is the construction:

Let $K$ be the field $k(z_\ell)$ if $\ell \neq 2$ and the field $k(i)$ if $\ell =2$.
It is a cyclic extension of $k$ of degree $t$ with Galois group $C_t$.
Let $T_1 = R_{K/k} \G_m$ be the torus: ``multiplicative group of $K$";
we have $T_1(k) = K^*$, and $T_1(k)$ contains the group $\<z_{_{\ell^m}}\>$, cf. 
\S 4.
If $\sigma$ is a generator of $C_t$, $\sigma$ acts on $T_1$, 
and we have $\sigma^t -1 = 0$ in the ring $\End(T_1)$.
Let us write the polynomial $X^t-1$ as $\Phi_t(X)\!\cdot\!\Psi(X)$, 
where $\Phi_t$ is the $t$-th cyclotomic polynomial.
We have $\Phi_t(\sigma) \Psi(\sigma)= 0$ in $\End(T_1)$.
Let $T = \Im \Psi(\sigma)$ be the image of 
$$\Psi(\sigma):T_1 \rightarrow T_1.$$ 

\noindent One checks that

(a) $\dim T = \varphi(t)$

(b) $T(k)$ contains $z_{_{\ell^m}}$.

\noindent
(For $\ell \neq 2$, (b) follows from the fact that the restriction of $\Psi(\sigma)$ 
to  $\<z_{_{\ell^m}}\>$ is an automorphism.
For $\ell=2$, use the fact that $T$ is the kernel of $\Phi_t(\sigma)$.)

Hence $T$ has the required properties.
\hfill$\Box$
\vskip 0.5cm
{\sl Alternate description of} $T$. It is enough to describe its character group $T^* = \Hom_{_{k_s}}(T,\G_m)$, together with the action of $\Gamma_k$ on $T^*$:\smallbreak

- $T^* = \Z [X]/\Phi_t(X)$ = algebraic integers of the cyclotomic field $\Q (\mu_t)$\vskip.1cm

- $\Gamma_k$ acts on $T^*$ by $\Gamma_k \rightarrow \Im \chi_{_{\ell^\infty}} \rightarrow C_t \,\,\tilde{\rightarrow}\,\, \Aut \big(\Q (\mu_t)\big).$\vskip.1cm

\noindent(It does not matter which isomorphism of $C_t$ onto $\Aut \big(\Q (\mu_t)\big)$ one chooses; they all give isomorphic tori.)

\subsection{The S-bound for reductive groups} 
Recall, cf. \S 3.3, that ``reductive" $\Rightarrow$ ``connected".

\begin{mytheorem5} 
{\sl Let} $G$ {\sl be a reductive group over} $k$, {\sl of rank} $r$, {\sl with Weyl group} $W$. {\sl If} $A$ {\sl is a finite subgroup of} $G(k)$, {\sl one has}
$$
v_\ell (A) \le m \left[ \frac{r}{\varphi (t)}\right] + v_\ell (W).
$$
\end{mytheorem5}
\begin{proof}
As usual, we may assume that $A$ is an $\ell$-group. In that case it is nilpotent, and by the corollary to th.3$^{\prime\prime}$ of \S 3.3 there exists a maximal $k$-torus $T$ of $G$ whose normalizer $N = N_G(T)$ contains $A$. Put $W_T = N/T$; this is a finite $k$-group such that $W_T(k_s)\simeq W$. If $A_T$ denotes the intersection of $A$ with $T(k)$, we have an exact sequence
$$
1 \rightarrow A_T \rightarrow A\rightarrow W_T(k)\, .
$$
Hence $v_\ell (A) \le v_\ell (A_T) + v_\ell \big(W_T(k)\big).$ By th.4, we have $v_\ell (A_T) \le m\cdot [r/\varphi (t)]$; on the other hand $W_T(k)$ is isomorphic to a subgroup of $W$, hence $v_\ell \big(W_T(k)\big) \hfill\break \le v_\ell (W)$. The theorem follows. \end{proof}

\begin{corollary}

If $r<\varphi(t)$, then $G(k)$ is $\ell$-torsion free {\rm (i.e. does not contain any elements of order} $\ell${\rm )}.

\end{corollary}
\begin{proof}
We have $\left[{\frac{r}{\varphi(t)}}\right]=0$. Hence by th.5 it is enough
to show that \linebreak
$v_{\ell}(W)=0$, but this follows from th.1 of \S 1.1 since $W$
is isomorphic to a subgroup of ${\bf GL}_r({\bf Z})$ and
$r<\varphi(t)\leq t\leq\ell-1$.
\end{proof}

\noindent {\sl Remark.} The ``S-bound" given by th.5 looks {\sl a priori} rather coarse:

(a) The torus $T$ is not an arbitrary torus of dimension $r$; the fact that it is a subtorus of $G$ puts non-trivial conditions on it; for instance the action of $\Gamma_k$ on $T^* = \Hom_{k_s} (T, \G_m)$ stabilizes the set of roots.

(b) The group $W_T(k)$ is in general smaller than $W$ itself, and the image of $N(k) \rightarrow W_T(k)$ may be even smaller.

It is therefore surprising how often the S-bound is close to being optimal. As an example, take $k = \Q$ and $G$ of type $E_8$. We have $m = 1$ and $t = \ell -1$ (except when $\ell = 2$ in which case $m = t = 2$), $r = 8$, $|W| = 2^{14}3^{5}5^{2}7$. The S-bound tells us that, if $A$ is a finite subgroup of $G(\Q )$, its order divides the number
$$
M_{S} = 2^{30}\!\cdot\!3^{13}\!\cdot\!5^6\!\cdot\!7^5\!\cdot\!13^2\cdot\!17
\!\cdot\!19\!\cdot 31\, .$$
We shall see later (cf. \S6.3.2 and \S7) that the best bound is $M = M_{S}/5\!
\cdot\!7\!\cdot\!17\,.$ In particular, the $\ell$-factors of $M_{S}$ are optimal for all $\ell$'s except $\ell = 5, 7$ and $17$. 
\vskip 1cm
%Paragraph 6
\begin{center}
{\bf {\S 6. The M-bound}}
\end{center}
\vskip 0.5cm
\setcounter{section}{6}
\setcounter{subsection}{0}
\subsection{Notation} From now on, $G$ is a semisimple \footnote{We could also
accept inner forms of reductive groups, for instance
${\bf GL}_n$ or more generally ${\bf GL}_D$, where $D$ is a central
simple $k$-algebra with $[D:k]=n^2$. In that case,
one has $r=n$, the $d_i$'s are the integers $1,2,\dots,n$ and th.6 is valid, with the 
same proof.}  group over $k$. We denote by $R$ its root system (over $k_s$), by $W$ its Weyl group, and by $r$ its rank. The group $W$ has a natural linear representation of degree $r$. The invariants of $W$ acting on $\Q[x_1,\dots ,x_r]$ make up a graded polynomial algebra of the form $\Q [P_1, \dots ,P_r]$, where the $P_i$ are homogeneous of degrees $d_i$, with \linebreak$d_1 \le d_2 \le \dots \le d_r$, (Shephard-Todd theorem, cf. e.g. [LIE V], \S5, th.4 or [Se 00], p.95). The $d_i$'s are called the {\sl invariant degrees} of $W$ (or of $G$). One has
$$
\prod d_i = |W| \quad {\mbox{and}} \quad \sum (2d_i-1) = \dim G\, .
$$
When $G$ is quasi-simple (i.e. when $R$ is irreducible) $d_r$ is equal to the Coxeter number $h = (\dim G)/r-1$, and one has the symmetry formula
$$
d_i + d_{r+1-i} = h +2\, .
$$
Moreover, if $j < h$ is prime to $h$, then $j+1$ is one of the $d_i$'s. These properties make $d_1,\dots , d_r$ very easy to compute (see e.g. the tables of [LIE VI]). 

For instance, for $G$ of type $E_8$, the $d_i$'s are: $2$, $8$, $12$, $14$, $18$, $20$, $24$, $30$. 

Let Dyn$(R)$ be the Dynkin diagram of $R$. There is a natural action of the Galois group $\Gamma_k$ on Dyn$(R)$: this follows from the fact that Dyn$(R)$ can be defined intrinsically from $G_{/k_s} $ (cf. [LIE VIII], \S 4, no 4, Scholie, or [SGA 3], expos\'{e} XXIV, \S3, p.344). In what follows (with the only exception of \S 6.7) we make the assumption that {\sl the action of } $\Gamma_k$ {\sl on} Dyn$(R)$ {\sl is trivial}: one then says that $G$ is {\sl of inner type} (it can be obtained from a split group $G_0$ by a Galois twist coming from  the adjoint group of $G_0$).
\vskip 0.3cm
\noindent {\sl Examples of groups of inner type{\rm{ :}}}

\noindent - ${\bf SL}_n$, or more generally, ${\bf SL}_D$, where $D$ is a central simple algebra over \nolinebreak$k$. 

\noindent - \!Any group $G$ whose root system has no non-trivial automorphism, e.g. any group of type $A_1, B_r, C_r, G_2, F_4, E_7, E_8$.

\subsection{Statement of the theorem} We fix $\ell, k$, and the root system $R$ of $G$. Recall that $\Im \chi_{_{\ell^\infty}}$ is a closed subgroup of $\Z^*_\ell$. Define:
$$
M(\ell , k,R) = \inf_{x\, \in\, \Im\, \chi_{_{\ell^\infty}}} \sum v_\ell (x^{^{d_i}}-1) = \inf_{g \,\in\, \Gamma_k} \sum v_\ell (\chi_{_{\ell^\infty}} (g)^{^{d_i}} - 1)\, .
$$
This is either an integer $\ge 0$ or $\infty$ (it is $\infty$ if and only if the invariants $m,t$ of $k$ defined in \S 4 are such that $m = \infty$ and $t$ divides one of the $d_i$'s, see prop.4 below).

\begin{mytheorem6}  Let $A$  be a finite subgroup of $G(k)$. Then $v_\ell (A) \le M\,({\ell , k,R)}$. {\rm (Recall that} $G$ {\rm is semisimple of inner type, cf. \S 6.1.)}\end{mytheorem6} \vskip0.2cm

This is what we call the ``M-bound" for $v_\ell(A)$. It will be proved in \S 6.5 below by a method similar to Minkowski's. We shall see in Lect. III that it is ``optimal" except possibly in the case $\ell = 2$, type (c) of \S 4.2.

For computations, it is useful to write $M$$(\ell ,k,R)$ explicitly in terms of the invariants $t$ and $m$ of \S 4:
\setcounter{proposition}{3}
\begin{proposition}
\rm{ 
(1) {\sl If} $\ell \not= 2$ {\sl or if} $\ell = 2, t = 1$ (case (a)), {\sl one has}
$$
M(\ell ,k,R) = \mathop{\sum_i}_{d_i \equiv \,\,0 \,\,( {\rm mod} \,t)} \big(m + v_{\ell}(d_i)\big)
$$
(2) {\sl If} $\ell = 2$ {\sl and} $t = 2$ \big(cases (b) and (c)\big), {\sl one has}
$$ M (2,k,R) = r_1 + mr_0 + v_2 (W)\, ,
$$
{\sl where} $r_0$ (resp. $r_1$) {\sl is the number of indices} $i$ {\sl such that} $d_i$ {\sl is even} (resp. $d_i$ {\sl is odd}).
}
\end{proposition}

\begin{proof}
Let us begin with the case $\ell \not= 2$. One shows first that, if $t|d$, one has $v_\ell (x^d - 1) \ge m + v_\ell (d)$ for every $x \in \Im\chi_{_{\ell^\infty}}$. (This is easy, since $x$ can be written as $zu$ with $z^t = 1$ and $v_\ell (u-1) \ge m$, hence $x^d-1 = u^d-1.)$

This already shows that $M(\ell ,k,R) \ge \sum_{t|d_i} \big(m + v_\ell (d_i)\big)$. To prove the opposite inequality, one chooses $x \in \Im \chi_{_{\ell^\infty}}$ of the form $zu$ with $z$ of order $t$ and $v_\ell(u\!-\!1) = m$. One gets (1).

The same argument works if $\ell = 2$ and $t = 1$. If $\ell = 2$ and $t = 2$, one has 
\begin{eqnarray*}
v_2 (x^d -1 ) & \ge & m + v_2 (d) \,\,\, {\mbox{ if $d$ is even}}\\
v_2(x^d -1) & \ge &1\,\,\, {\mbox {if $d$ is odd}}\, ,
\end{eqnarray*}
for every $x \in \Im\chi_{_{2^\infty}}$. This gives:
$$
M(2,k,R) \geq \sum_{d_i\,\,{\rm odd}} 1 + \sum_{d_i\,\,{\rm even}} \big(m + v_2 (d_i)\big) = r_1 + mr_0 + v_2 (W)\, .
$$
To get the opposite inequality, observe that $x = -1 + 2^m$ belongs to $\Im \chi_{_{2^\infty}}$ and check that $\sum v_2 (x^{^{d_i}} -1)$ is equal to $r_1 + mr_0 + v_2(W)\, .$
\end{proof}

\begin{corollary}{\rm{
{\sl Let} $a(t)$ {\sl be the number of indices $i$ such that} $d_i \equiv 0$ (mod $t$). {\sl If} $a(t) = 0,$ {\sl then} $G(k)$ {\sl is} $\ell$-{\sl torsion free}. 
\vskip0.1cm
Indeed, if $a(t) = 0$, the sum occurring in prop.4 is an empty sum, hence $M(\ell ,k,R) = 0$ and one applies th.6.
\hfill $\Box$
}}\end{corollary}

\subsection{Two examples: $A_1$ and $E_8$} We take $k = \Q$, so that $t = \ell-1$ and $m=1$ if $\ell > 2$ and $t = m = 2$ if $\ell = 2$.

\subsubsection{Type $A_1$} There is only one $d_i$, namely $d_1 = 2$, and prop.4 gives:
$$
M\big(\ell ,\Q,A_1\big) = \left\{
\begin{array}{lll}
3&{\mbox{if}}&\ell = 2\\
1&{\mbox{if}}& \ell = 3\\
0&{\mbox{if}}& \ell > 3\, .
\end{array}
\right.
$$
In other words, every finite subgroup of $G(\Q)$ has an order which divides $2^3\!\cdot \!3$. This bound is optimal in the following sense:

(a) The split adjoint group $\P\G\L_2 (\Q)$ contains both a subgroup of order 3 and a dihedral subgroup of order $8$ (but no subgroup of order 24).

(b) The simply connected group ${\bf SL}_{\bf H} (\Q)$, where $\H$ is the standard quaternion division algebra, contains a subgroup of order $24$ which 
is isomorphic to ${\bf SL}_2 (\F_3)$. However the split group ${\bf SL}_2 (\Q)$
does not contain any subgroup of order 8 (but it does contain cyclic subgroups of order 3 and 4).

\subsubsection{Type $E_8$} If we define $M(\Q,E_8)$ as $\prod_\ell \ell^{^{M(\ell , \Q,E_8)}}$, prop.4 gives:
$$
M(\Q ,E_8) = 2^{30}\!\cdot \!3^{13}\!\cdot \!5^5\!\cdot\! 7^4\!\cdot\! 11^2\!
\cdot\! 13^2\!\cdot \!19\!\cdot\! 31,\,\, {\mbox{see e.g. [Se 79], \S 3.3}}.
$$
By th.6, the order of every finite subgroup of $G(\Q)$ divides $M(\Q,E_8)$. As we shall see in the next lecture, this multiplicative bound is optimal.

\subsection{A Chebotarev-style result} We need such a result in order to generalize Minkowski's method of \S 1. 

Let $L$ be a normal domain which is finitely generated over $\Z$ as a ring, and let $k$ be its field of fractions. If $d = \dim (L)$ denotes the Krull dimension of $L$ ([AC VIII], \S 1), one has ({\sl{loc.cit.}}, \S 2):
$$
\begin{array}{rclcrcl}
d&= &1 + {\rm tr.deg} (k/\Q)&{\mbox{if}}&\car(k)& = & 0\\
d &=&{\rm tr.deg} (k/\F_p)&{\mbox{if}}&\car(k)&= &p > 0\,.
\end{array}
$$

Let Specmax$(L)$ be the set of the maximal ideals of $L$ (= set of closed points of Spec$(L)$). If $x \in$ Specmax$(L)$, the residue field $\kappa (x) = L/x$ is finite (see e.g. [AC V], p. 68, cor. 1). We put $Nx = |\kappa(x)|$; it is the {\sl norm} of $x$.

When $d = 0$, $L$ is a finite field, and Specmax$(L)$ has only one element. If $d > 0$ (e.g. when $\car (k) = 0$), then Specmax$(L)$ is infinite. More precisely, the Dirichlet series $z(s) = \sum_x 1/(Nx)^s$ converges for Re$(s) >d$, and one has
\begin{equation} 
z(s) \sim \log \big(1/(s-d)\big) \quad {\mbox{when}}\quad s \rightarrow d\quad ({\mbox{with}}\,\, s >d)\, .
\end{equation}
See [Se 65], \S 2.7, which only contains a sketch of proof; complete details (for a slightly weaker statement) can be found in [Pi 97], App. B
%begin Fussnote
\footnote{When $\car(k) = 0$\,\, one can give a stronger statement, in the spirit of the Prime Number Theorem: 

For every $X \ge 2$, call $\pi_L(X)$ the number of $x \in$ Specmax$(L)$ such that $Nx \le X$. Then:
$$
\pi_L(X) = (1/d) \,X^d\!/\log X + O (X^d\!/\log^2 X)\quad {\mbox{when}}\quad X \rightarrow \infty \, .$$
The general Chebotarev density theorem can also be stated (and proved) in terms of such ``natural" density (standard method: use Weil-Deligne estimates to reduce everything to the known case $d = 1$). }; see \linebreak
also [FW 84], pp.206-207.
%end Fussnote

\vskip0.2cm
Let now $n$ be an integer $\ge 1$ which is invertible in $L$ (and hence in $k$). Let $\chi_n : \Gamma_k \rightarrow (\Z/n\Z)^*$ denote the $n$-th cyclotomic character of $k$, cf. \S 4.0. As in \S 4, we shall be interested in $\Im \chi_n 
\subset (\Z/n\Z)^*$.
 
{\begin{mytheorem7}
\label{thm7} Let $c$ be an element of $(\Z/n\Z)^*$, and let $X_c$ be the set of all $x \in$ {\rm Specmax}$(L)$ such that $Nx \equiv c$ {\rm (mod $n$)}. Then {\rm{:}}

{\rm a)} If  $c\notin \Im \chi_n$, then $X_c = \varnothing\,.$

{\rm b)} If $c\in \Im \chi_n$ and $d > 0$, then $X_c$ is Zariski-dense in {\rm Specmax}$(L)$ {\rm (or in Spec}$(L)$, {\rm this amounts to the same).} In particular, $X_c$ is infinite.\vskip0.2cm

 {\rm A more concrete formulation of b) is that, for every non-zero $f \in L$, there exists an $x$ with $f \not\in x$ and $Nx \equiv c$ (mod $n$).}
\end{mytheorem7}}

\noindent {\sl Example.} Take $L = \Z [1/n]$. Then Specmax$(L)$ is the set of all prime numbers which do not divide $n$, and th.7 translates into Dirichlet's theorem on the existence of primes in arithmetic progressions.

\subsection*{\sl Proof of theorem 7.} The group $C = \Im \chi_n$ is the Galois group of the cyclotomic extension $k(z_n)/k$. Let $L_n$ be the integral closure of $L$ in $k(z_n)$. One checks by standard arguments that the ring extension $L_n/L$ is finite and \'{e}tale. In geometric terms, Spec$(L_n) \rightarrow$ Spec$(L)$ is a finite \'{e}tale covering. The group $C$ acts freely on Spec$(L_n)$, with quotient Spec$(L)$. For every closed point $x$ of Spec$(L)$, the Frobenius element $\sigma_x$ of $x$ is a well-defined conjugacy class of $C$ (hence an element of $C$ since $C$ is commutative). Moreover, if we view $C$ as a subgroup of $(\Z/n\Z)^*$, $\sigma_x$ is the image of $Nx$ in $\Z/n\Z$. This proves a).
 
 Suppose now that $d > 0$ and that $c$ belongs to $C = \Im \chi_n$. Let $z_c(s)$ be the Dirichlet series $\sum 1/(Nx)^s$, where the sum is over the elements $x$ of $X_c$. The general Chebotarev density theorem ([Se 65], [Pi 97]) gives:
 \begin{equation}\label{eq642}
 z_c (s) \sim \frac{1}{|C|} \log (1/(s-d))\quad {\mbox{when}}\,\,\, s \rightarrow d\quad {\mbox{with}} \quad s > d\, .
 \end{equation}
 In particular, we have $z_c(d) = + \infty$. If the Zariski closure $\overline{X}_c$ of $X_c$ were of dimension $< d\!-\!1$, we would have $z_c(d) < \infty$, as one sees by splitting $\overline{X}_c$ into irreducible components, and applying (6.4.1). Hence b).
 \hfill$\Box$
 
 \subsection{\bf Proof of theorem 6} Let $A \subset G(k)$ be as in th.6. We want to prove that 
 $$
 v_\ell (A) \le M(\ell ,k,R)\, .$$
 We do it in three steps:
 
 \subsubsection{\bf The case where $k$ is finite} Put $q = |k|$. It is well-known that 
 $$
 |G(k)| = q^N \prod (q^{^{d_i}} - 1),\quad\quad {\mbox{where}}\quad N = |R|/2 = \sum (d_i - 1).
 $$
 This shows that $v_{\ell}(A) \le \sum v_\ell (q^{^{d_i}}-1)$. The integer $q$, viewed as an element of $\Z^*_\ell$, is a topological generator of $\Im \chi_{_{\ell^\infty}}$. Hence every element $u$ of $\Im \chi_{_{\ell^\infty}}$ is an $\ell$-adic limit of powers of $q$ and this implies that $v_\ell (u^d - 1) \ge v_\ell (q^d - 1)$ for every $d \ge 1$. Hence the lower bound which defines $M(\ell ,k,R)$ is equal to $\sum v_\ell (q^{^{d_i}}-1)$; this proves th.6 in the
case where $k$ is finite.
 
 \subsubsection{\bf The case where $k$ is finitely generated over its prime subfield} By 6.5.1, we may assume that $k$ is infinite. We need a subring $L$ of $k$, with field of fractions $k$, which has the following properties:

(a) $L$ is normal, finitely generated over $\Z$ and contains $1/\ell$.
 
(b) $G$ comes by base change from a semisimple group scheme $\underline{G}$ over $L$, in the sense of [SGA 3], XIX. 2.7.
 
(c) $A$ is contained in the group $\underline{G}(L)$ of the $L$-points of $\underline{G}$.

 \begin{lemma} There exists such an $L$.
 
{\rm This is standard, see e.g. [EGA IV], \S 8.1}\hfill$\Box$
  \end{lemma}
 Let us now choose $(L,\underline{G})$ with properties (a), (b) and (c). For every \linebreak$x \in $ Specmax$(L)$, the fiber $\underline{G}_x$ of $\underline{G}$ at $x$ is a semisimple group over $\kappa (x)$, of type $R$. Moreover, the Dynkin diagram of $\underline{G}$ is finite \'{e}tale over Spec$(L)$, cf. [SGA 3], XXIV.3.2; since it is ``constant" for the generic fiber (i.e. over $k$) it is constant everywhere; this shows that the $\underline{G}_x$ are of inner type. The inclusion map $i: A \rightarrow \underline{G}(L)$ gives for every $x$ a homomorphism\linebreak
$i_x : A \rightarrow \underline{G}\big(\kappa (x)\big)$. Since $i$ is injective, there is an open dense subset $X_0$ of Specmax$(L)$ such that $i_x$ is injective for all $x \in X_0$. We thus get:
$$
v_\ell (A) \,\, \le\,\, v_\ell \big(\underline{G} (\kappa (x))\big) = \sum v_\ell \big((Nx)^{^{d_i}}-1\big)\quad {\mbox{for all}} \quad x \in X_0,$$
cf. 6.5.1. Let $u$ be any element of $\Im \chi_{_{\ell^\infty}}$. By applying th.7 to the image of $u$ in $(\Z/\ell^j\Z)^*$ with $j = 1,2,\dots ,$ we find a sequence of points $x_j$ of $X_0$ such that $\lim Nx_j = u$ in $\Z^*_\ell$. We have:
$$
v_\ell (u^{^{d_i}} -1) = \lim_{j \rightarrow \infty} \sum v_\ell \big((Nx_j)^{^{d_i}}-1\big)\, ,
$$
and applying the formula above to each of the $x_j$'s we obtain
$$
v_\ell(A) \le \sum v_\ell (u^{^{d_i}} - 1)\quad\quad {\mbox{for every}} \quad u \in \Im \chi_{_{\ell^\infty}}\, .
$$
This proves th.6 in the case 6.5.2.
\vskip 0.2cm
[{\small Variant: One reduces the general case to the case where $\dim(L) = 1$ by using Hilbert's irreducibility theorem, as explained in [Se 81], p.2; in the case $\dim(L) = 1$, one can apply the standard Chebotarev theorem instead of the general one.}]

\subsubsection{\bf The general case} The same argument as for lemma 7 shows that $G$ comes by base change from a semisimple group $G'$ over a subfield $k'$ of $k$ which is finitely generated over the prime subfield of $k$ (i.e. $\F_p$ or $\Q$). Moreover, one may assume (after enlarging $k'$ if necessary) that $A$ is contained in $G'(k')$. The Galois group $\Gamma_{k'}$ acts on the Dynkin diagram Dyn$(R)$ of $G'$ (which is the same as the one of $G$). Let $k''$ be the Galois extension of $k'$ corresponding to the kernel of $\Gamma_{k'} \rightarrow \Aut$ Dyn$(R)$. Since $G$ is of inner type over $k$, the field $k''$ is contained in $k$. By base change to $k''$, $G'$ gives a semisimple group $G''$ which is of inner type and we may apply 6.5.2 to $(G'',A)$. We get $v_\ell (A) \le M(\ell ,k'',R)$. Since $k''$ is contained in $k$, we have $M(\ell ,k'',R) \le M(\ell ,k,R)$ : the group $\Im \chi_{_{\ell^\infty}}$ can only decrease by field extensions. Hence $v_\ell (A) \le M(\ell ,k,R)$. \hfill$\Box$

\subsubsection{\bf Remark} Surprisingly, the proof above does not really use the hypothesis that $A$ is a subgroup of $G(k)$. It uses only that $A$ {\sl acts freely on} $G$, viewed merely as a $k$-variety (and not as a homogeneous space); this is indeed enough to ensure that $v_\ell (A) \le v_\ell (G(k))$ when $k$ is finite. Here is an example: take $G = {\bf SL}_2$, $\ell = 2$, $k = \Q$; the M-bound is 3, which means that any finite 2-subgroup of ${\bf SL}_2 (\Q)$ has order $\le 8$. As was said in \S 6.3.1, there is in fact no subgroup of order 8 in ${\bf SL}_2(\Q)$. But one can make a cyclic group of order 8 act freely on the variety ${\bf SL}_2$: take for instance the group generated by the automorphism:
$$
\begin{pmatrix}a&b\cr c&d\cr\end{pmatrix} \mapsto \begin{pmatrix}d-c&-c-d\cr
{\frac{(a-b)}{2}}&{\frac{(a+b)}{2}}\cr\end{pmatrix}=\begin{pmatrix} 0&-1\cr 
{\frac{1}{2}}&0\cr\end{pmatrix}\begin{pmatrix} a&b\cr c&d\cr\end{pmatrix}\begin{pmatrix} 1&1\cr -1&1\cr\end{pmatrix}.$$ 

Hence, even in this bad-looking case, the M-bound can claim to be ``optimal".

\subsection{An analogue of Sylow's theorem}
\begin{mytheorem8} Let $A$ and $A'$ be two finite $\ell$-subgroups of $G(k)$. Assume that $v_\ell (A)$ is equal to the {\rm M-}bound $M(\ell ,k,R)$. Then there exists $y \in G(\bar{k})$ such that $yA'y^{-1} \subset A.$
\end{mytheorem8}

\begin{corollary}
If both $A$ and $A'$ attain the {\rm M-}bound, then they are geometrically conjugate {\rm (i.e. conjugate in $G(\bar{k}))$.} In particular, they are isomorphic.
\end{corollary}

\begin{proof}
We may assume that $k$ is finitely generated over its prime subfield. If it is finite, th.8 is just a special case of Sylow's theorem. Let us assume that $k$ is infinite, and choose $L, \underline{G}$ as in \S 6.5.2 with $A,A' \subset \underline{G}(L)$. Let $Y$ be the subscheme of $\underline{G}$ made up of the points $y$ with $yA'y^{-1} \subset  A$. Let $X$ be the set of all $x \in$ Specmax$(L)$ such that $Nx$, viewed as an element of $\Z^*_\ell$, is of the form $z_tu$ with $z_t$ of order $t$ and $v_\ell (u\!-\!1) = m$ (note that $m$ is finite, cf. \S 4.3). It follows from th.7, applied to $n = \ell^{m+1}$, that $X$ is Zariski-dense in Spec$(L)$. If $x \in$ Specmax$(L)$, the groups $A$ and $A'$ inject into $\underline{G} (\kappa (x))$ (this is an easy consequence of the hypothesis that $\ell$ is invertible in $L)$. If moreover $x$ belongs to $X$, then the same computation as in \S 5.2 shows that $v_\ell \big(\underline{G}(\kappa (x)\big)$ is equal to the M-bound, hence $A$ is an $\ell$-Sylow of $\underline{G}(\kappa (x))$. By Sylow's theorem, this shows that $A'$ is conjugate in $\underline{G}(\kappa (x))$ to a subgroup of $A$. In particular, the fiber at $x$ of $Y \rightarrow$ Spec$(L)$ is non-empty. Since $X$ is Zariski-dense, this implies that the generic fiber $Y_{/k}$ of $Y \rightarrow$ Spec$(L)$ is non-empty, i.e. that $Y\!(\bar{k})$ is non-empty.
\end{proof}

\noindent{\sl Remark.} One can show that $Y$ is smooth over $L$, and hence that $Y(k_s) \not= \varnothing$ which is slightly more precise than $Y\!(\bar{k}) \not= \varnothing$.
\vskip 0.3cm
\noindent{\small{\sl Exercise.} Show that a family of polynomial equations with coefficients in $\Z$ has a solution in $\C$ if and only if it has a solution in $\Z/p\Z$ for infinitely many $p$'s.}

\subsection{Arbitrary semisimple algebraic groups} In the previous sections, we have assumed that $G$ is of inner type, i.e. that the natural homomorphism
$$
\varepsilon : \Gamma_k \rightarrow \Aut {\rm Dyn}(R)
$$
is trivial. Let us now look briefly at the general case, where no hypotheses on $\varepsilon$ are made. In order to state the result which replaces th.6 we need to introduce the linear representations $\varepsilon_d$ of $\Gamma_k$ defined as follows:

Let $S = \Q[P_1,\dots,P_r]$ be the $\Q$-algebra of $W$-invariant polynomials, cf. \S 6.1. Let $I = (P_1, \dots , P_r)$ be the augmentation ideal of $S$; put $V = I/I^2$. The vector space $V$ is of dimension $r$, and is graded; the dimension of its $d$-th component $V_d$ is equal to the number of indices $i$ with $d_i = d$. The group $\Aut {\rm Dyn}(R)$ acts on $S$, $V$ and the $V_d$'s; by composing this action with $\varepsilon$, we get for each $d > 0$ a linear representation
$$
\varepsilon_d : \Gamma_k \rightarrow \Aut (V_d)\, .
$$
\vskip 0.3cm
\noindent{\bf Theorem 6$^\prime$.} {\sl Let $A$ be a finite subgroup of $G(k)$. Then}:
$$
v_\ell (A) \le \inf_{g\,\in\, \Gamma_k} \sum_{d} v_\ell \big(\det (\chi_{_{\ell^\infty}} (g)^d - \varepsilon_d(g))\big)$$
(The determinant is relative to the vector space $V_d \otimes \Q_\ell\, .$)
\vskip 0.5cm
\noindent {\sl Proof (sketch)}. The method is the same as the one used for th.6. There are three steps:

(1) Reduction to the case where $k$ is finitely generated over its prime subfield; this is easy.

(2) Reduction to the case where $k$ is finite, via the general Chebotarev density theorem instead of th.7.

(3) The case where $k$ is finite. In that case, if $q = |k|$, and if $\sigma$ is the Frobenius generator of $\Gamma_k$, one has (cf. e.g. [St 68] th. 11.16)

$$
v_\ell \big(G(k)\big) = \sum_d v_\ell \big(\det (q^d - \varepsilon_d(\sigma))\big) = \sum_d v_\ell \big(\det (\chi_{_{\ell^\infty}} (\sigma )^d - \varepsilon_d(\sigma ))\big)$$
hence the desired formula:
$$
(\ast )\quad\quad v_\ell (A) \le \sum_d v_\ell \big(\det (\chi_{_{\ell^\infty}}(g)^d - \varepsilon_d(g))\big)
$$
in the special case $g = \sigma$. By applying this to the finite extensions of $k$, one sees that the inequality $(\ast )$ is valid for all $\sigma^n , n = 1,2,\dots ,$ and hence for all $g \,\in \,\Gamma_k$, since the $\sigma^n$ are dense in $\Gamma_k$.\hfill$\Box$
\vskip 0.2cm
\noindent{\sl Remark.} One may also prove th.6$^\prime$ using $\ell$-adic cohomology, cf. \S 6.8.
\vskip 0.2cm
\noindent{\sl Example.} Take $R$ of type $A_2$, so that  $\Aut$ Dyn$(R) = \{1,-1\}$ and $\varepsilon$ may be viewed as a quadratic character of $\Gamma_k$. The $V_d$'s are of dimension $1$ for $d = 2,3$ and are 0 otherwise. The action of $\Aut$ Dyn$(R)$ on $V_d$ is trivial for all $d$, except $d = 3$. Hence $\varepsilon_2 = 1$, $\varepsilon_3 = \varepsilon$, and th.6$^\prime$ can be rewritten as:
$$
v_\ell (A) \,\le\, \inf_{g\in\Gamma_k} \left\{ v_\ell (\chi_{_{\ell^\infty}} (g)^2 - 1) + v_\ell \big(\chi_{_{\ell^\infty}}(g)^3 - \varepsilon (g)\big)\right\}\, .
$$
A similar result holds for the types $A_r$ ($r > 2$), $D_r$ ($r$ odd) and $E_6$, with 2 (resp. 3) replaced by the even $d_i$'s (resp. the odd $d_i$'s).

\subsection{The cohomological method} Let us consider first the general situation suggested in \S 6.5.4 where a finite group $A$ acts freely on a quasi-projective $k$-variety $X$. As explained in [Il 05], \S 7, one can then give an upper bound for $v_\ell (A)$ in terms of the action of $\Gamma_k$ on the \'{e}tale cohomology of $X$. More precisely, let $H^i_c(X)$ denote the $i$-th \'{e}tale cohomology group of  $X_{/k_s}$, with proper support and coefficients $\Q_\ell$; it is a finite dimensional $\Q_\ell$-vector space which is 0 for $i > 2\!\cdot
\!\dim(X)$. There is a natural action of $\Gamma_k$ on $H^i_c(X)$, and, for each $g \in \Gamma_k$, one can define the ``Lefschetz number" $\Lambda_X(g)$ by the usual formula:
$$
\Lambda_X(g) = \sum_i (-1)^i {\rm Tr}\big(g |H^i_c(X)\big)\, .
$$
One has $\Lambda_X (g) \in \Z_\ell$. Moreover:
\vskip 0.5cm
\noindent{\bf Theorem 6$^{\prime\prime}$.} $v_\ell (A) \le \inf_{g\,\in\,\Gamma_k} v_\ell \big(\Lambda_X(g)\big)\, .$

\begin{proof}
See [Il 05], \S 7, especially cor.7.5. The proof follows the same pattern as the other proofs of the present \S: one uses Chebotarev to reduce to the case where $k$ is finite, in which case the result follows from the fact, due to Grothendieck, that, if $\sigma$ is the (geometric) Frobenius generator of $\Gamma_k$, then $\Lambda_X(\sigma )$ is equal to $|X(k)|$, hence is divisible by $|A|$ since the action of $A$ is free. (As in the proof of th.6$^\prime$, one applies this, not only to $\sigma$ but also to its powers $\sigma^n$, $n > 0$, and one uses the fact that the $\sigma^n$ are dense in $\Gamma_k$\,.)
\end{proof}

If one applies th.6$^{\prime\prime}$ to $A \subset G(k)$, with $A$ acting by left translations on $X = G$, one recovers th.6 and th.6$^\prime$, thanks to the known structure of the cohomology of $G$, cf. e.g. [SGA 4$\frac{1}{2}$], p. 230.

\subsection{The Cremona group: open problems} Recall that the {\sl Cremona group} {\bf Cr}$_r(k)$ is the group of $k$-automorphisms of the field $k(X_1,\dots,X_r)$, i.e. the group of birational automorphisms (or ``pseudo-automorphisms", cf. [De 70]) of the projective $r$-space over $k$. For $r = 1$, one has ${\bf Cr}_1(k) = \P\G\L_2(k)$. Let us assume that $r \ge 2$. As explained in [De 70], ${\bf Cr}_r$ is not an algebraic group, but looks like a kind of very large semisimple group of rank $r$ (very large indeed: its ``Weyl group" is the infinite group $\G\L_r(\Z)$). Not much is known about the finite subgroups of ${\bf Cr}_r(k)$ beyond the classical case $r = 2$ and $k$ algebraically closed. Here is a question suggested by \S 5.1:

- Is it true that ${\bf Cr}_r(k)$ has no $\ell$-torsion if $\varphi (t) > r$? 

\noindent In the special case $k = \Q$, $r = 2$ or $3$, this amounts to:

- Is it true that the fields $\Q(X_1,X_2)$ and $\Q(X_1, X_2, X_3)$ have no automorphism of prime order $\ge 11$? (Automorphisms of order $2$, $3$, $5$ and $7$ do exist.)

It would be very interesting to attack these questions using cohomology, but I do not see how to do this. It is not even clear how to define cohomological invariants of ${\bf Cr}_r(\C)$, since there is no natural topology 
on that group. Still, one would like to give a meaning to a sentence such as 
$$``{\bf Cr}_r(\C) {\hbox{\rm { is connected for }}} r \ge 1 {\hbox{\rm { and simply-connected for }}} r \ge 2."$$

% [{\it Note added Nov.1, 2010.} %For partial answers to these questions, see:

%I.V.Dolgachev and V.A.Iskovskikh, {\it On elements of prime order in the plane Cremona group}, arXiv:0707.4305.

%J-P.Serre, {\it Le groupe de Cremona et ses sous-groupes finis}, S\'em. Bourbaki 2008-2009, no 1000; AstŽrisque 332, SMF (2010), 75-100.

%%%%%%%%%%%%%%%%%%%%%%%%%%%%%%%%%%%%%%%%%%%%%%%%%%%%%%%%%%%%%%%%%%%%%%%
%%%%%%%%%%%%%%%%%%%%%%%%%%%%%%%%%%%%%%%%%%%%%%%%%%%%%%%%%%%%%%%%%%%%%%%%
%\include{serredonna_3}

%\markboth{J.-P. SERRE, ORDERS OF FINITE SUBGROUPS}{bidon}

\markright{LECTURE III: CONSTRUCTION OF LARGE SUBGROUPS}
%\markleft{J.-P. SERRE, ORDERS OF FINITE SUBGROUPS}

 \specialsection*{\bf III. Construction of large subgroups}

\setcounter{section}{8}
\numberwithin{equation}{section}
\numberwithin{equation}{subsection}
\setcounter{subsubsection}{0}
\setcounter{footnote}{2}
\vskip 0.5cm
\begin{center}
{\bf {\S 7. Statements}}
\end{center}
\label{sec1}
We keep the notation of Lecture II: $k$, $\ell$, $\chi_{_{\ell^\infty}}$, $t$, $m$, \ldots.
We consider only semisimple groups over $k$ with a root system $R$ which is {\em irreducible.}
The M-bound of \S 6.2 will be denoted by $M(\ell,k,R)$; 
it only depends on the pair $(\ell,k)$ via the invariants $t$ and $m$, and on $R$ via the degrees $d_1, \ldots, d_r$ of $W$.
We limit ourselves to the case $m < \infty$; see \S 14 for the case  $m=\infty$.

A pair $(G,A)$, where $G$ is of inner type with root system $R$, and 
$A \subset G(k)$ is a finite group, will be called {\em optimal} if $v_\ell(A)$ is equal to the M-bound $M(\ell,k,R)$.
(We could assume that $A$ is an $\ell$-group, but this would not be convenient for the constructions which follow.)
Our goal is to prove:

\begin{mytheorem9}
\label{thm3.1}
If $\ell \neq 2$, an optimal pair $(G,A)$ exists {\rm (}for any $k$, $R${\rm )}.
\end{mytheorem9}

\begin{mytheorem10}
\label{thm3.2}
If $\ell = 2$, an optimal pair $(G,A)$ exists if $\Im \chi_{_{2^\infty}}$ is 
of type {\rm (a)} or {\rm (b)} in the sense of {\rm \S 4.2}
\emph{
(i.e. if $\Im  \chi_{_{2^\infty}}$ can be topologically generated by one element).}
\end{mytheorem10}

\begin{mytheorem11}
\label{thm3.3}
In the case $\ell=2$ and type {\rm (c)}, there exists $(G,A)$ with 
$$v_2(A) = r_0 m + v_2(W)$$
where $r_0$ is the number of indices $i$ such that $d_i$ is even.
\end{mytheorem11}

Note that here the M-bound is $M(2,k,R)=r_1+r_0m+v_2(W)$ with $r_1=r-r_0$, cf. \S6.2, prop.4. Hence $v_2(A)$
differs from $M(2,k,R)$ only by $r_1$. In particular, $A$ is optimal if $r_1=0$. Hence:

\begin{corollary}
%\label{cor1}
If all the $d_i$'s are even {\rm (i.e. if }$-1 \in W$), then an optimal pair $(G,A)$ exists for $\ell = 2$ {\rm (and hence for all $\ell$'s, thanks to th.9).

\vskip.1cm

\noindent This applies in particular to the exceptional types $G_2$, $F_4$, $E_7$ and $E_8.$}
\end{corollary}

\noindent {\sl Remarks}. (1) The simplest case where the M-bound is not attained is $k = \Q$, $\ell =2$, $R$ of type $A_2$,
where $m=2$, $r_0=1$, $r=2$, the $M$-bound is 4, and it follows from [Sch 05]
that $v_2(A) \leq 3$ for every finite subgroup $A$ of $G(\Q)$.

\smallskip
 
(2) In Theorems 9, 10 and 11,  no claim is made on the structure of  $G$ except that it is of inner type and that its root system is of type $R$. However, if one looks closely at the proofs given in the next sections, one sees that $G$ can be chosen to have the following properties:

\noindent - it is simply connected;

\noindent - it splits over the cyclotomic field $k(z_\ell )$ if $\ell >2$, and over $k(i)$ if $\ell = 2.$

Simple examples (such as $k=\Q$, $\ell =3$, $G$ of type $G_2)$ show that it is not always possible to have $G$ split over $k$.

\smallskip

(3) If $G$ is not chosen carefully, the group $G(k)$ may not contain
any large $\ell$-subgroup, even if $k$ contains
all the roots of unity. For instance, when $R$ is of type $A_1$ (resp. of type $E_8$) it is easy (resp. it is possible) to construct a pair $(G,k)$ such that
the only torsion elements of $G(k)$ have order $1$ or $2$ (resp. $G(k)$ is torsion free).
\smallskip

(4) The three theorems above are almost obvious if the characteristic is $p \neq 0$
(especially Theorem 11 since type (c) never occurs!):
one takes a finite field $k_0$ contained in $k$ which has the same invariants $t$ and $m$
(this is easily seen to be possible -- if $k$ is finitely generated over $\F_p$, one chooses
the maximal finite subfield of $k$), 
and one takes for $G$ the group deduced by base change from a split group $G_0$ over $k_0$
with root system $R$.
If we choose for $A$ the finite group $G_0(k_0)$,  it is clear from the way we got the 
M-bound that $v_\ell(A) = M(\ell,k_0,R) = M(\ell,k,R)$, so that $(G,A)$ is optimal.

\vskip0.1cm

In what follows, we shall assume that $\car(k) =0$.
Note also that we could replace $k$  by any subfield having the same invariants $t$ and $m$, 
for instance the intersection of $k$ with the field of $\ell^\infty$-roots of unity.
We could thus assume that {\em $k$ is a cyclotomic number field}, if needed.\vskip.1cm

The proof of Theorem 9 will be given first for classical groups (\S 9),
by explicit elementary constructions similar to those of Schur.
The more interesting case of exceptional groups (\S 12) will use different methods, based on 
Galois twists (\S 10), Tits groups and braid groups (\S 11). 
The case of $\ell=2$ will be given in \S 13. The last section (\S 14) is about $m = \infty$.

% 8. Arithmetic methods
\vskip 0.5cm
\begin{center}
{\bf {\S 8. Arithmetic methods}} $(k = \Q)$
\end{center}
\vskip 0.5cm

These methods are not strong enough to prove the statements of \S 7,
but they give very interesting special cases.

\subsection{\bf Euler characteristics} 

Here, the ground field is $\Q$.
One starts from a split simply connected group scheme $G$ over $\Z$ (this makes sense, cf.~\cite{SGA3}).
One may thus speak of the group $\Gamma = G(\Z)$ of the {\em integral points} of $G$.
It is a discrete subgroup of $G(\R)$.
Its Euler characteristic $\chi (\Gamma)$ (``caract\'eristique d'Euler-Poincar\'e'' in French) is well-defined
(see [Se 71] and [Se 79]); it is a rational number.
Moreover it is proved in [Ha 71] that 
%equation 8.1.1.
\begin{equation}
\label{eq3.1}
\chi(\Gamma) = c \prod_{i=1}^r \frac{1}{2} \zeta(1-d_i) 
= c \prod_{i=1}^r \frac{b_{d_i}}{2d_i},
\end{equation}
where $b_d$ is the $d$-th Bernoulli number,
$\zeta$ is the zeta function and\linebreak
\mbox{$c = |W|/|W_K|$} where 
$W_K$ is the Weyl group of a maximal compact subgroup $K$ of $G(\R)$.
Assume that all $d_i$'s are {\em even} (if not, all the terms in (\ref{eq3.1}) are zero).
Using standard properties of Bernoulli numbers, one can check that 
{\em the {\rm M}-bound relative to $\ell$ is 
$M=\sum_i v_\ell \big ( \den \big ( \frac{1}{2} \zeta(1-d_i) \big) \big)$}, where ``den" means denominator. Hence, if $\ell$ does not divide $c$, and does not divide the numerator of any $\frac{1}{2} \zeta(1-d_i)$
(which is the case if $\ell$ is a so-called regular prime), 
one sees that {\em the denominator of  
{\rm EP}$(\Gamma)$ is divisible by $\ell^M$.}
But a theorem of K.~Brown~\cite{Br74} shows that this is only possible if $\Gamma$ contains a 
finite subgroup of order $\ell^M$.
Hence we get an optimal pair (provided $(c,\ell)=1$, and $\ell$ is regular, say).
\smallskip

\noindent{\sl Example}. Take $G$ of type $E_8$; here $c=3^3\!\cdot\!5$, and the numerators of the $\frac{1}{2} \zeta(1-d_i)$
do not cancel any denominator.
Hence one obtains that a split $E_8$ contains an optimal $A$ for all $\ell \neq 3,5$, with
the extra information that $A$ can be found inside the group $\Gamma = G(\Z)$ -- 
but no information on what it looks like!

\subsection{\bf Mass formulae} 

In \cite{Gr96}, B. Gross considers $\Q$-forms of $G$ such that $G(\R)$ is {\em compact}; 
he also  requires another condition which guarantees that $G$ has a {\em smooth model over $\Z$}.
This condition is fulfilled for types $B$, $D$, $G_2$, $F_4$ and $E_8$.
He then proves a {\em mass formula} \`a la  Minkowski (\cite{Gr96}, prop.2.2):
$$
\sum \frac{1}{|A_\sigma|} = \prod_{i=1}^r \frac{1}{2} \zeta(1-d_i)
$$
where the $A_\sigma$ are the $\Z$-points of the smooth models of $G$ over $\Z$
(taken up to conjugation).
Each $A_\sigma$ is finite.
It is then clear that, if $\ell^N$ is the $\ell$-th part of the denominator of $\prod_{i=1}^r \frac{1}{2} \zeta(1-d_i)$, 
the $\ell$-Sylow subgroup of one of the $A_\sigma$ has order $\geq \ell^N$.
If $N$ is equal to the Minkowski bound $M$ 
(which happens if $\ell$ does not divide the numerator of any of the $\frac{1}{2} \zeta(1-d_i)$),
then such a Sylow has order $\ell^M$, and we get an optimal pair.
Note that there is no extra factor ``$c$'' as in (\ref{eq3.1}).
This works very well for $G_2$, $F_4$, $E_8$ 
(and some classical groups too, cf. \cite{Gr96}):

\noindent $G_2$ - 
	Here the mass is $\frac{1}{4} \, \zeta(-1)\zeta(-5) = \displaystyle \frac{1}{2^63^37}$,
	and it is obtained with just one $A_\sigma$, which turns out to be isomorphic to $G_2(\F_2)$.
	\vskip3mm

\noindent $F_4$ -
	There are two $A_\sigma$'s and the mass formula is 
	$$\begin{array}{rcl} \displaystyle
	\frac{1}{2^{15} \cdot 3^6 \cdot 5^2 \cdot 7} + \frac{1}{2^{12} \cdot 3^5 \cdot 7^2 \cdot 13} 
	&= &\frac{1}{16} \, \zeta(-1)  \zeta(-5)  \zeta(-7)  \zeta(-11) \\
	&= & \displaystyle \frac{691}{2^{15} \cdot 3^6 \cdot 5^2 \cdot 7^2 \cdot 13}.
	\end{array}$$
	\vskip2mm
\noindent $E_8$ -
	Here the numerator is very large, but the denominator is exactly what is needed 
	for the M-bound, namely:
	$$2^{30} \cdot  3^{13} \cdot 5^5 \cdot 7^4 \cdot 11^2 \cdot 13^2 \cdot 19 \cdot 31.$$

%--------------------        Paragraph 9           --------------------          

\vskip 0.5cm
\begin{center}
{\bf {\S 9. Proof of theorem 9 for classical groups}}
\end{center}
\vskip 0.5cm
\setcounter{section}{9}
\setcounter{subsection}{0}

Here $\ell \neq 2$. 
Recall that $\Im \chi_{_{\ell^\infty}} = C_t \times \{1 + \ell^m \Z_\ell\}$, 
where $m \geq 1$ and $t$ divides $\ell-1$.
The M-bound is 
$$
M = \mathop{\sum_i}_{d_i \equiv 0 (\mod t)} \big ( m + v_\ell(d_i) \big ).
$$
We denote by $K$ the field $k(z_\ell)$ generated by a root of unity of order $\ell$.
It is a cyclic extension of $k$, of degree $t$, with Galois group $C_t$. 
It contains $z_{_{\ell^m}}$ but not $z_{_{\ell^{m+1}}}$, cf. \S 4.1.

%--------------------           9.1 The groups A_N and            --------------------          

\subsection{\bf The groups $A_N$ and $A_N^{\,1}$}
\label{subsec3.1}

If $N$ is an integer $\geq 1$, we denote by $A_N$ the subgroup of $\GL_N(K)$
%in der nåÉchsten Linie habe ich \ell statt l eingefåÄgt
(where $K=k(z_\ell)$ as above) generated by the symmetric group $S_N$ and
the diagonal matrices whose entries are $\ell^m$-th roots of unity
(wreath product of $S_N$ with a cyclic group of order $\ell^m$).
We have
\begin{equation}
v_\ell(A_N)= mN + v_\ell(N!).
\end{equation}
The image of $\det_K : A_N \rightarrow K^*$ is $\{\pm 1\} \times \< z_{_{\ell^m}}\>$.
Hence the kernel $A_N^{\,1}$ %of that map 
is such that
\begin{equation}
\label{eq3.5}
v_\ell(A_N^{\,1})= m(N-1) + v_\ell(N!).
\end{equation}
We are going to use $A_N$, and sometimes $A_N^{\,1}$, in order to construct
optimal subgroups for the classical groups $\SL_n$, $\SO_n$ and $\Sp_n$; this is what Schur did in \cite{schur}, \S 6, for the case of $\G\L_n$.

%--------------------           subsection 9.2          --------------------          

\subsection{\bf The case of $\SL_n$}

The $d_i$'s are $2$, $3$,$\ldots$, $n$. 
If we put $N =  \left[ \frac{n}{t} \right ]$, we have
\begin{equation}
\label{eq3.8}
\qquad M =  mN + v_\ell(N!) \qquad \mbox{ if } t \geq 2, \qquad  \qquad \qquad \qquad 
\end{equation}
\begin{equation}
\label{eq3.9}
\qquad M =  m(N-1) + v_\ell(N!) \qquad \mbox{ if } t = 1, \mbox{ in which case } N=n.
\end{equation}

In the case $t \geq 2$, we take $A_N \subset \GL_N(K) \subset \GL_{Nt}(k)$, and
observe that $\det_k(A_N)$ is equal to $\pm1$
(indeed, if $g \in A_N$, then 
\mbox{$\det_k(g) = N_{K/k}\big (\det_K(g) \big )$} 
and one checks that
$ N_{K/k}\big ( z_{_{\ell^m}} \big ) = 1$).
This shows that an $\ell$-Sylow of $A_N$ is contained in $\SL_{Nt}(k)$ and 
hence in $\SL_n(k)$. 
By~(\ref{eq3.8}) we get an optimal pair.

In the case $t=1$, we use the same construction with $A_N^{\,1}$ instead of $A_N$.
The comparison of~(\ref{eq3.5}) and~(\ref{eq3.9}) shows that we get an optimal pair.

%--------------------           subsection 9.3           --------------------          

\subsection{\bf The case of the orthogonal and symplectic groups, $t$ odd}

Let us consider the case of $\Sp_{2n}$. The $d_i$'s are equal to $2,4,\dots ,2n$. Hence, 
if we put $N =  \left[ \frac{n}{t} \right ]$, the M-bound is $mN + v_\ell(N!)$. 
There is a natural embedding:
$$\GL_N\rightarrow \Sp_{2N} \rightarrow \Sp_{2n}$$
defined by $x \mapsto \begin{pmatrix} x & 0 \\ 0 & \,^tx^{-1} \end{pmatrix}$.
The image of $A_N$ by that embedding is optimal.

The same construction works for $\SO_{2n}$ and  $\SO_{2n+1}$.
(Note that, in all these cases, we get the {\em split} forms of the groups of type $B_n$, $C_n$, $D_n$.
This is no longer true 
%will not be true any more 
in the case $t$ is even -- nor in the cases of  \S 12.)

%--------------------           subsection 9.4           --------------------          

\subsection{\bf The case of the orthogonal and symplectic groups, $t$ even}

Since $t$ is even, the group $C_t = \Gal(K/k)$ contains an element $\sigma$ 
of order 2; its image in $\Z_\ell^*$ is $-1$.
Let $K_0$ be the subfield of $K$ fixed by $\sigma$;
we have $[K\!:\!K_0] = 2$, $[K_0\!:\!k] = t_0$ with $t_0 = t/2$.
Moreover $\sigma(z_{_{\ell^m}})$ is equal to $ (z_{_{\ell^m}})^{-1}$;
i.e. $\sigma$ acts on $ z_{_{\ell^m}}$ just as complex conjugation does.
Let us define an {\em hermitian form} $h$ on $K^N$ 
(where $N$ is a given integer $\ge 1$) by the standard formula
$$
\quad
h(x,y) = \sum_{i=1}^N x_i\!\cdot\!\sigma(y_i), \qquad 
\mbox{ if }ÌÉx = (x_1, \ldots, x_N),\, y = (y_1, \ldots, y_N).
$$
If $\U_N$ denotes the {\em unitary group} associated with $h$, it is clear that 
{\em the group $A_N$ defined in {\rm \S 9.1} is contained in $\U_N(K)$.} [We use here the traditional notation $\U_N (K)$ for the unitary group; this is a bit misleading, since $\U_N$ is an algebraic group over $K_0$, and we are taking its $K_0$-points.]

Let $\delta \in K^*$ be such that $\sigma(\delta) = - \delta$, e.g. $\delta = z_{\ell} - z_{\ell}^{-1}$.
We have $K = K_0 \oplus \delta\!\cdot\!K_0$, and $h(x,y)$ can be decomposed as
$$
%$$
\quad  
h(x,y) = q_0(x,y) + \delta\cdot b_0(x,y), \,\,{\mbox{with}}\,\, q_0(x,y) \in K_0, \quad b_0(x,y) \in K_0.
%$$
$$
Then $q_0$ (resp. $b_0$) is a non-degenerate symmetric (resp. alternating) 
$K_0$-bilinear form of rank $2N$. 

\noindent Its trace $q = \Tr_{K_0/k} q_0(x,y)$ (resp. \mbox{$b = \Tr_{K_0/k} b_0(x,y)$)} is of rank
$2Nt_0 = Nt$ over $k$.
We thus get embeddings:
\begin{equation}
A_N \rightarrow \U_N(K) \rightarrow \SO_{2N}(K_0) \rightarrow  \SO_{Nt}(k) 
\end{equation}
\begin{equation}
A_N \rightarrow \U_N(K) \rightarrow \Sp_{2N}(K_0) \rightarrow  \Sp_{Nt}(k).
\end{equation}
Now, for a given $n$, let us define $N$ by 
$N =  \left[ \frac{2n}{t} \right ]=  \left[ \frac{n}{t_0} \right ]$.
By (9.4.2), we get an embedding 
$$
A_N  \rightarrow  \Sp_{Nt}(k) \rightarrow  \Sp_{2n}(k),
$$
and one checks that it is optimal.

The same method gives an embedding of $A_N$ into $\SO_{Nt}(k)$, hence into $\SO_{2n+1}(k)$, and this embedding is also optimal.
As for $\SO_{2n}(k)$, one has to be more careful. The method does give an embedding of $A_N$ into the $\SO_{2n}$ group relative to some quadratic form $Q$, but we have to ensure that such an $\SO_{2n}$ group is {\em of inner type} i.e.  that $\disc(Q) = (-1)^n$ in $k^*/{k^*}^2$. There are three cases:

a) If $2n >Nt$ (i.e. if $t$ does not divide $2n$), we choose 
$Q  = q \oplus q_1$, where $q_1$ has rank $2n-Nt$, and is such that 
$\disc(q) \cdot \disc(q_1) = (-1)^n$. We then have 
$ A_N \subset \SO_{2n,Q}(k)$ and this is optimal.

b) If  $2n =Nt$ and $N$ is even,  we have $\disc(q) = d^N$,  where
$d = \disc(K_0/k)$, hence $\disc(q)$= $1$ in $k^*/{k^*}^2$, which is the same as $(-1)^n$
since $n$ is even. 

c) If $2n=Nt$ and $N$ is odd, we use an optimal subgroup $A$ of  $ \SO_{2n-1}(k)$ relative to a quadratic form $q_0$ of rank $2n-1$. By adding to $q_0$ a suitable quadratic form of rank 1, we get a quadratic form of rank $2n$ and discriminant $(-1)^n$, as wanted. The corresponding embedding 
$$A \rightarrow \SO_{2n-1}(k) \rightarrow \SO_{2n}(k)$$ is optimal. (Note that the $d_i$'s for type $D_n$ are $2,4,\dots, 2n-2$, and $n$.
Hence, if $t\!\not| \, n$, the M-bound for $D_n$ is the same as the M-bound for  $B_{n-1}.)$

%--------------------           Paragraph 10           --------------------          

\vskip 0.5cm
\begin{center}
{\bf {\S 10. Galois twists}}
\end{center}
\vskip 0.5cm

To handle exceptional groups, we have to use {\em twisted} inner forms 
instead of split ones.
We shall only need the most elementary case of twisting, namely the one coming from a
homomorphism $\varphi:\Gamma_k \rightarrow \Aut(G)$.
Let us recall what this means (cf. for example \cite{Se64}, chapter~III):

Let $K/k$ be a finite Galois extension.
Let $X$ be an algebraic variety over $k$, assumed to be quasi-projective
(the case where $X$ is affine would be enough).  Choose a homomorphism
$$\varphi: \Gal(K/k) \rightarrow \Aut_k X.$$
The {\em twist} $X_\varphi$ of $X$ by $\varphi$ is a variety over $k$ which can be 
characterized as follows:

There is a $K$-isomorphism $\theta\!:\!X_{/K}\rightarrow {X_\varphi}_{/K}$ such that $\gamma(\theta) = \theta \circ \varphi(\gamma)$
for every $\gamma \in \Gal(K/k)$.

(Here $X_{/K}$ denotes the $K$-variety deduced from $X$ by the base change 
\mbox{$k \rightarrow K$,}
and $\varphi(\gamma) \in \Aut_k X$ is viewed as belonging to $\Aut_{K}X_{/K}$.)

One shows (as a special case of Galois descent) that such a pair $(X_\varphi,\theta)$ exists,
and is unique, up to isomorphism.

It is sometimes convenient to identify the $K$-points of $X$ and $X_\varphi$ 
via the isomorphism $\theta$.
But one should note that this is not compatible with the natural action of $\Gal(K/k)$
on $X(K)$ and $X_\varphi(K)$;  one has
$$
\qquad 
\gamma \big( \theta(x) \big) = \varphi(\gamma) \big ( \gamma(x) \big )
	 \qquad \mbox{ if }ÌÉ\gamma \in \Gal(K/k), x \in X(K). 
$$
In other words, if we identify $X_\varphi(K)$ with $X(K)$, an element $\gamma$ of  $\Gal(K/k)$
acts on $X_\varphi(K)$ by the {\em twisted action{\rm { :}}}
$$
x \mapsto \varphi(\gamma) \big ( \gamma(x) \big )
$$
In particular, the  {\em $k$-rational points of $X_\varphi$} correspond (via $\theta^{-1}$) 
to the points $x \in X(K)$ such that 
$\gamma(x) = \varphi(\gamma^{-1}) x$ for every $\gamma \in \Gal(K/k)$.

In what follows we apply the $\varphi$-twist to $X=$ split form of $G$,
with $\varphi(\gamma)$ being a $k$-automorphism of $G$ for every $\gamma \in \Gal(K/k)$.
In that case, $G_\varphi$ is a $k$-form of $G$; this form is inner if  all $\varphi(\gamma)$
belong to $G^{\ad}(K)$ where $G^{\ad}$ is the adjoint group of $G$.
The effect of the twist is to make $k$-rational some elements of $G$ which were not.
In order to define $\varphi$, we shall have to use the $k$-automorphisms of $G$ provided by 
the Tits group $W^*$, see next section.

%--------------------           SECTION 11          --------------------          

\vskip 0.5cm
\begin{center}
{\bf {\S 11. A general construction}}
\end{center}
\vskip 0.5cm
\setcounter{section}{11}
\setcounter{subsection}{0}

Here, $G$ is a split simply connected group over $k$, 
and $T$ is a maximal split torus of $G$.
We put $N=N_G(T)$ and $W = N/T$ is the Weyl group.

\subsection{\bf The Tits group}

The exact sequence $1 \rightarrow T \rightarrow N \rightarrow W \rightarrow 1$
does not split in general.
However Tits (\cite{Ti66a}, \cite{Ti66b}) has shown how to construct a 
subgroup\footnote{The construction of $W^*$ depends on more than $(G,T)$:
one needs a \emph{pinning} (``\'epinglage'') of $(G,T)$ in the sense 
of~\cite{SGA3}, XXIII.1.1. 
%\marginpar{\blue precise ref in SGA3?}
} $W^*$ of $N(k)$ having the following properties:

(1) The map $W^* \rightarrow W$ is surjective.

(2) The group $W^* \cap T$ is equal to the subgroup $T_2$ of $T$
made up of the points $x$ of $T$ with $x^2 = 1$.

We thus have a commutative diagram, where the vertical maps are inclusions:
$$
\begin{array}{ccccccccc}
1 & \rightarrow & T_2 & \rightarrow & W^* & \rightarrow & W & \rightarrow & 1 \\
\downarrow & & \downarrow & & \downarrow \, & & \downarrow & & \\
1 & \rightarrow & T & \rightarrow & N \, & \rightarrow & W & \rightarrow & 1 \\
\end{array}
$$
We refer to Tits (\emph{loc. cit.}) and to Bourbaki\footnote{Bourbaki works in the context of compact
real Lie groups; his results can easily be translated to the algebraic setting we use here.}
([LIE X], pp. 115--116, exerc. 12, 13) for the construction and the properties of $W^*$. 
For instance:

If $G$ comes from a split group scheme $\underline{G}$ over $\Z$, 
then $W^*$ is equal to $\underline{N}(\Z)$, 
the group of \emph{integral points} of the group scheme $\underline{N}$.

In the case of $\SL_n$, this means that one can choose for $W^*$ the group of monomial matrices
with non-zero entries $\pm 1$ and determinant $1$. 
For $n=2$, $W^*$ is the cyclic group of order 4 generated by $\begin{pmatrix} 0&1\\-1&0 \end{pmatrix}$.

Note also that $W^*$  is a quotient of the {\emph{braid group}}  ${\bf B}_W$ associated to $W$.
(For the definition of the braid group of a Coxeter group, see e.g.~\cite{BM97}.)

\subsection{\bf Special elements of $W$}

We now go back to our general notation $\ell ,m,t,\ldots$ of Lecture II. 
Recall that the M-bound $M =M(\ell,k,R)$ is given by
\begin{equation}
\label{11.2.1} 
M = \sum_{t | d_i} \left( m + v_\ell(d_i) \right), \qquad \mbox{ cf. \S 6.2}.
\end{equation}
Let  $a(t)$ be the number of indices $i$ such that $d_i \equiv 0 \, (\mod t)$.
We may rewrite  (\ref{11.2.1}) as 
\begin{equation}
M = m a(t) + \sum_{t|d_i} v_\ell(d_i).
\end{equation}

Note that, if no $d_i$ is divisible by $t$, we have $M=0$ 
and the trivial group $A=1$ is optimal.
Hence \emph{we shall assume in what follows that $a(t) \geq 1$.}
\vskip0.1cm
Let now $w$ be an element of $W$.
We shall say that $w$ is {\emph {special}} (with respect to $t$ and $\ell$)
if it has the following four properties:

\noindent (1) $w$ has order $t$ in $W$.

\noindent (2) $w$ is the image of an element $w^*$ of $W^*$ such that
$(w^*)^t \in T_2 \cap C(G)$, where $C(G)$ is the center of $G$.

\noindent (3) The characteristic polynomial of $w$ 
(in the natural $r$-dimensional representation of $W$)
is divisible by $(\Phi_t)^{a(t)}$, where $\Phi_t$ is the $t$-th cyclotomic polynomial.
\newline
(Equivalently: if $z_t$  denotes a primitive $t$-th root of unity, 
then $z_t$ is an eigenvalue of $w$ of multiplicity at least $a(t)$.)

\noindent (4) Let $C_W(w)$ be the centralizer of $w$ in $W$.
Then:
$$ v_\ell\left (C_W(w) \right) \geq \sum_{t|d_i} v_\ell(d_i).$$

\noindent{\sl Remark}. The reader may wonder whether special elements exist for
a given pair $(t,\ell)$ (with $a(t)>0$ and $\ell\equiv 1\ (\mod t)$, of course). The
answer is ``no'' in general: if $R$ is of type $C_3$ and $t=4$, no element
of $W^*$ has both properties (1) and (2). Fortunately, the answer is ``yes'' for the exceptional types $G_2,\dots,E_8$, cf. \S 12.\medbreak

\noindent {\sl Example{\rm{ :}} the regular case}.
Suppose that $w \in W$ is \emph{regular of order} $t$ in the sense
of Springer\footnote{With a slight difference: Springer requires $t > 1$ and we don't; it is convenient to view $w = 1$ as a regular element of $W$.

Note that, if $t$ is given, there is a very simple criterion ensuring the existence of a regular element of $W$ of order $t$: the number of indices $i$ such that $d_i \equiv 0$ (mod $t$) should be equal to the number of $i$'s such that $d_i \equiv 2$ (mod $t$), cf. Lehrer-Springer \cite{LS99}, cor.5.5.}
 (\cite{Sp74}, bottom of p. 170 - see also \cite{BM97}, \S 3). This means that $w$ has an eigenvector $v$, with eigenvalue $z_t$, such that $v$ does not belong to any reflecting hyperplane. {\sl Then} $w$ {\sl is special} (for any $\ell$ with $\ell \equiv 1$ (mod $t$)). Indeed:
 
 (1) is obvious.
 
 (2) follows from the fact, proved in \cite{BM97}, \S 3, that $w$ has a lifting ${\mathbf w}$ in the braid group ${\mathbf B}_W$ with 
${\mathbf w}^t = \mbox{\boldmath$\pi$}$, where {\boldmath$\pi$} has an image $\pi$ in $W$ which belongs to $T_2 \cap C(G)$. In Bourbaki's notation ([LIE X], p.116) $\pi$ is the canonical element $z_G$ of the center of $G$.
 
 %this citation above  is not existing 
(3) is proved in \cite{Sp74}, th. 4.2.
 
(4) is proved in \cite{Sp74}, th. 4.2, in the stronger form $|C_W(w)|= \prod_{t|d_i} d_i$.
\vskip 0.3cm
\noindent{\sl Special cases}

$t = 1$. Here $w = 1$ and $w^* = \pi$ (one could also take $w^* = 1$).

$t =\!2$. Here $w = w_0 =$ longest element of $W$. When $-1$ belongs to $W$, one has $w_0 = -1$ and $w^*_0$ is {\sl central}  in $W^*$ (because ${\mathbf w}_0$ is central in ${\mathbf B}_W$, cf. \cite{BM97}, 1.2 and 3.4). In that case the inner automorphism of $G$ defined by $w^*_0$ is a ``Weyl-Chevalley involution": it acts on $T$ by $t \mapsto t^{-1}$.

\setcounter{lemma}{6}
\subsection{An auxiliary result}

\begin{lemma}
\label{lem7}
Suppose $w \in W$ is special of order $t$.
Then it is possible to choose a lifting $w^*$ of $w$ in $W^*$
which satisfies\vskip0.1cm 

{\em (2*)} $(w^*)^t \in T_2 \cap C(G)$

\noindent and 

{\em (4*)} $v_\ell \left ( C_{W^*}(w^*) \right ) \geq \sum_{t|d_i} v_\ell(d_i)$.
\end{lemma}

\begin{proof}
Let $P$ be an $\ell$-Sylow of $C_W(w)$; 
the groups $P$ and $\langle w \rangle$ commute, 
and $P \cap \langle w \rangle=1$ since $w$ has order $t$
and $\ell$ is prime to $t$ (since $\ell\equiv 1 \, \mod t$).
Hence the group $P_w$ generated by $w$ and $P$ 
is the direct product $P \times \langle w \rangle$.
Since $\ell \neq 2$, its 2-Sylow subgroup is contained in $\langle w \rangle$.
Put $C_2 = T_2 \cap C(G)$. 
We have an exact sequence:
$$
\label{eq48}
1 \rightarrow T_2/C_2 \rightarrow W^*/C_2 \rightarrow W \rightarrow 1.
$$
By property (2) of $w$, this exact sequence splits over $\langle w \rangle$, 
hence over the 2-Sylow of $P_w$;
since the order of $T_2/C_2$ is a power of 2, this implies that it splits over $P_w$.
We thus get an element $w'$ of $W^*/C_2$, of order $t$, which lifts $w$, 
and centralizes a subgroup $P'$ of $W^*/C_2$ isomorphic to $P$.
We then choose for $w^*$ a representative of $w'$ in $W^*$; it has
property (2$^*$), moreover 
its centralizer contains the inverse image of $P'$, 
which is canonically isomorphic to $C_2 \times P'$.
By property (4) we have
$$v_\ell(P') = v_\ell(P) \geq \sum_{t|d_i} v_\ell(d_i).$$
This shows that $w^*$ has property (4*).
\end{proof}

\noindent{\sl Remark}.
In the case where $w$ is regular, one can do without lemma~\ref{lem7}.
Indeed the braid group construction of~\cite{BM97}
gives a lifting $w^*$ of $w$ having property (2$^*$) and such that the map
$C_{W^*}(w^*) \rightarrow C_W(w)$ is surjective.

\subsection{The main result}
\setcounter{proposition}{4}
\begin{proposition}
Suppose $W$ contains an element $w$ which is special 
with respect to $t$ and $\ell$. 
Then there exist an inner twist $G_\varphi$ of $G$ {\rm(cf. \S 10)}
and a finite $\ell$-subgroup $A$ of $G_\varphi(k)$ such that the pair
$(G_\varphi, A)$ is optimal in the sense of \S $7$.

{\em(In particular, th.9 is true for $(k,\ell,R)$.)}
\end{proposition}

\begin{proof}
As in \S 9, we put $K = k(z_\ell)$, where $z_\ell$ is a root of unity
of order $\ell$.
Let $C_t = \Gal(K/k)$; it is a cyclic group of order $t$.

Choose $w^* \in W^*$ with the properties of lemma~\ref{lem7} and
let $\sigma$ be the inner automorphism of $G$ defined by $w^*$.
Since $\sigma$ has order $t$, there exists an injective homomorphism:
$$\varphi: C_t \rightarrow G^{\ad}(k) \subset \Aut_k(G)$$
which maps $C_t$ onto the subgroup $\langle \sigma \rangle$ of $\Aut_k(G)$
generated by $\sigma$.
As explained in \S 10, we may then define the \emph{$\varphi$-twist}
$G_\varphi$ of $G$, relatively to the Galois extension $K/k$. 
The group $G_\varphi$ is an inner form of $G$;
it has the same root system $R$.
It remains to construct a finite $\ell$-subgroup $A$ of $G_\varphi(k)$ 
such that $(G_\varphi,A)$ is optimal,
i.e. $v_\ell(A) = m a(t) +  \sum_{t|d_i} v_\ell(d_i)$, cf. (11.2.2).

We take for $A$ the semi-direct product $E_m\cdot P$, with 
$E_m \subset T_{\varphi}(k)$ and $P \subset N_\varphi(k)$, where $E_m$ and $P$
are defined as follows:\vskip.1cm

(1) $P$ is an $\ell$-Sylow of $C_{W^*}(w^*)$. 
By lemma~\ref{lem7} we have $v_\ell(P) \geq \sum_{t|d_i} v_\ell(d_i)$.
\newline
Note that the points of $P$ are fixed by $\sigma$.
Hence these points are rational over $k$ not only in the group $G$
but also in the group $G_\varphi$.

(2) $E_m$ is the subgroup of $T_\varphi(k)$ made up of the elements $x$ 
such that $x^{\ell^m}=1$.

\noindent It is clear that $P$ normalizes $E_m$, and that $P \cap E_m = 1$.

\begin{lemma}
\label{lem8}
The group $E_m$ contains a product of $a(t)$ copies of the group $\Z/\ell^m\Z$.
\end{lemma}

This implies that $v_\ell(E_m) \geq m a(t)$ and hence
$$v_\ell(A) = v_\ell(E_m)+v_\ell(P) \geq m a(t) +\sum_{t|d_i} v_\ell(d_i).$$
We thus get $v_\ell(A)\geq M$ and since $M$ is an upper bound for $v_\ell(A)$
we have $v_\ell(A)=M$.
\end{proof}

\noindent{\sl Proof of lemma 8}.
Consider first the subgroup $T_{\ell^m}$ of $T(k_s)$ made up of the elements  $x$ 
with $x^{\ell^m}=1$. 
Since $T$ is $k$-split,  and $K = k(z_\ell) = k(z_{\ell^m})$
(cf.  \S 4 and \S 9), 
the points of $T_{\ell^m}$ are rational over $K$.
If we write $T_{\ell^m}(K)$ additively, it becomes a free $\Z/\ell^m\Z$-module of rank $r$
and the action of a generator $s$ of $C_t$ is by $x \mapsto sx$,
where $s$ is identified with an element of order $t$ in $\Z_\ell^*$
(i.e. $s = $``$z_t$'' with our usual notation for roots of unity).
As for the action of $w^*$ (i.e. of $w$) on $T_{\ell^m}(K)$, 
it can be put in diagonal form 
since $w$ is of order $t$ and $t$ divides $\ell-1$;
its diagonal elements are $r$ elements $y_1, \ldots, y_r$ of $\Z/\ell^m\Z$,
with $y_i^t=1$.
Let $c$ be the largest integer such that $(\Phi_t)^c$ divides the 
characteristic polynomial of $w$.
By property (3) of  11.2, we have $c \geq a(t)$
(in fact, $c=a(t)$, by \cite{Sp74}, th.~3.4).
This implies that the family of the $y_i$'s contains $c$ times each primitive
$t$-th root of unity (viewed as element of $(\Z/\ell^m\Z)^*$).
In particular, there is a $\Z/\ell^m\Z$-submodule $X$ of $T_{\ell^m}(K)$
which is free of rank $c$
and on which $w$ acts by $x \mapsto z_t^{-1} x$.
If we twist $G$, $T$, $T_{\ell^m}$ by $\varphi$,
the new action of $C_t = \Gal(K/k)$ on $X$ is trivial
(cf. end of \S 10).
This means that $X$ is contained in $T_\varphi(k)$, hence in $E_m$,
which proves the lemma.

\hfill$\Box$
\vskip 0.3cm

Note the following consequence of proposition 4:

\begin{corollary}
If $W$ contains a $t$-regular element in the sense of \cite{Sp74},
then theorem 9 is true for $k,\ell,R$.
\end{corollary}

In the case $t=1$, no twist is necessary
(one takes $w=1$, $w^*=1$, cf. \S 11.2).

%--------------------          section 12          --------------------          

\vskip 0.5cm
\begin{center}
{\bf {\S 12. Proof of theorem~9 for exceptional groups}}
\end{center}
\setcounter{section}{12}
\setcounter{subsection}{0}

In each case we will show that the Weyl group contains an element $w$
which is special with respect to $t$ and $\ell$, so that we may apply 
prop.5.

\subsection{The case of $G_2$}

The degrees $d_i$ are $d_1=2, d_2=6$.
Since $t$ divides one of them, $t$ is a divisor of $6$, 
hence is regular (\cite{Sp74}, no. 5.4).
We may then apply prop.5.
\qed

Explicit description of $w,w^*$: if $c$ is a Coxeter element of $W$,
$c$ is of order $6$, and every lifting $c^*$ of $c$ in $W^*$ has order 6.
Hence, for any divisor $t$ of 6, we may take $w=c^{6/t}$ and $w^* = (c^*)^{6/t}$.

\subsection{The case of $F_4$}

The $d_i$'s are: 2, 6, 8, 12. 
All their divisors are regular (Springer, \emph{loc. cit.}).
One concludes as for $G_2$.
\qed

\subsection{The case of $E_6$}

The $d_i$'s are: 2, 5, 6, 8, 9, 12.
All their divisors are regular, except $t=5$.
In that case, choose any element $w \in W$ of order 5.
Since the kernel of $W^* \rightarrow W$  is a 2-group, 
$w$ can be lifted to an element $w^*$ of $W^*$ of order 5.
Conditions (1) and (2) of \S11.2 are obviously satisfied.
The same is true for condition (3), since $a(5)=1$ 
(only one of the $d_i$'s is divisible by 5),
and $w$ has at least one eigenvalue of order 5.
As for condition (4), it is trivial, since $\ell\equiv 1 \, (\mod 5)$ implies $\ell\geq 11$,
and $\ell$ does not divide any of the $d_i$'s, so that 
$\sum_{t|d_i} v_\ell(d_i)$ is 0.
Hence $w$ is special with respect to $(5,\ell)$.
\qed

\subsection{The case of $E_7$}

The $d_i$'s are: 2, 6, 8, 10, 12, 14, 18.
By~\cite{Sp74}, {\it loc.cit.} all their divisors are regular except 4, 5, 8, 10, 12.
If $t=4, 5, 8$ or 12, $t$ already occurs for $E_6$, 
with the same values of $a(t)$, namely 2, 1, 1 and 1.
Hence, we have $E_6$-special elements $w_4, w_5, w_8$ and $w_{12}$ in $W(E_6)$.
One then takes their images in $W(E_7)$ by the injective map $W(E_6) \rightarrow W(E_7)$,
and one checks that they are $E_7$-special (here again condition (4) is trivial since 
$v_\ell(d_i)=0$ for all the $\ell$'s with $\ell \equiv 1 \, (\mod t)$).

As for $t=10$, one takes $w = -w_5$, which makes sense since $-1\in W$.
The element $-1$ (usually denoted by $w_0$)  can be lifted to 
a central element $\varepsilon$ of $W^*$ with $\varepsilon^2 \in T_2 \cap C(G)$;
this is a general property of the case $-1\in W$ (which reflects the fact that $-1$ is $2$-regular, see end of \S 11.2).
Hence, if $w_5^*$ is a lifting of $w_5$ of order 5, $\varepsilon w_5^*$  is a lifting of
$w$ of order 10, and this shows that $w$ is special with respect to 10 and $\ell$.
\qed

\subsection{The case of $E_8$}

The $d_i$'s are: 2, 8, 12, 14, 18, 20, 24, 30.
By~\cite{Sp74}, {\it loc.cit.}, all their divisors are regular except 7, 9, 14, 18.

If $t=7$ (resp. 9), one chooses $w_7  \in W$ of order 7 (resp. $w_9 \in W$ of order 9).
Since 7 and 9 are odd, condition (2) of \S 11.2 is satisfied. 
The same is true for condition (3) because $a(t)=1$, and for condition (4) because $v_{\ell}(d_i)=0$ for all $i$.

If $t=14$ (resp. 18), one takes $w=-w_7$ (resp. $w=-w_9$), 
as we did for $E_7$.
\qed

\vskip 0.5cm
\begin{center}
{\bf {\S 13. Proof of theorems 10 and 11}}
\end{center}
\vskip 0.5cm
\setcounter{section}{13}
\setcounter{subsection}{0}

Here $\ell$ = 2. There are three cases (cf. \S 4.2): 

(a) $\Im\chi_{_{2^\infty}} = 1 + 2^m {\mathbf Z}_2 $ with $m \ge 2$.
In that case the M-bound is\linebreak $rm + v_2 (W)$, and th.10 asserts that an optimal pair $(G,A)$ exists for every type $R$.

(b) $\Im\chi_{_{2^\infty}} =\<-1 + 2^m\>$, with $m \ge 2$. The M-bound is $r_0m + r_1 + v_2(W)$, where $r_0$ (resp. $r_1$) is the number of $i$'s such that $d_i$ is odd (resp. even). Here, too, th.10 asserts that an optimal pair exists. 

(c) $\Im \chi_{_{2^\infty}} = \<-1,1+2^m\>$, with $m \ge 2$.

The M-bound is the same as in case (b), but th.11 does not claim that it can be met (i.e. that an optimal pair exists); it merely says that there is a pair $(G,A)$ with $v_2(A) = r_0m + v_2(W)$; such a pair is optimal only when $r_1 = 0$, i.e. when $-1$ belongs to the Weyl group.

\subsection{Proof of theorem 10 in case (a).} We take $G$ split and simply connected, and we choose a maximal split torus $T$. We use the notation $(N,W,W^*)$ of \S 11. Let $E$ be the $2$-torsion subgroup of $T(k)$. Since $T$ is isomorphic to the product of $r$ copies of $\G_m$, $E$ is isomorphic to a product of $m$ copies of $\Z/2^m\Z$, cf. \S 4.2. Hence $v_2(E) = rm$. The group $E$ is normalized by the Tits group $W^*$; we define $A$ as $A = E\!\cdot\!W^*$. The exact sequence
$$
1 \rightarrow E \rightarrow A \rightarrow W \rightarrow 1
$$
shows that $v_2(A) = rm + v_2(W)$. Hence $(G,A)$ is optimal.

\subsection{Cases (b) and (c).} As in \S11.4, we start with a split $G$, with a split maximal torus $T$. We define $N, W, W^*$ as usual. After choosing an order on the root system $R$, we may view $W$ as a Coxeter group; let $w_0$ be its longest element. It has order 2, and it is regular in the sense of Springer \cite{Sp74}. As explained in \S 11.2, this implies that there is a lifting $w^*_0$ of $w_0$ in $W^*$ which has the following two properties:

(i) its square belongs to the center of $G$;

(ii) the natural map $C_{W^*}\big(w^*_0\big) \rightarrow C_W(w_0)$ is surjective. 

\noindent Let $\sigma$ be the inner automorphism of $G$ defined by $w^*_0$. By (i), we have $\sigma^2 = 1$. Let $K = k(i)$ and let $\varphi$ be the homomorphism of Gal$(K/k)$ into Aut$_k(G)$ whose image is $\{1,\sigma\}$. Let us define $G_\varphi$ as the $\varphi$-{\sl twist} of $G$, in the sense defined in \S 10. Denote by $T_\varphi ,N_\varphi$ and $W^*_\varphi$ the $\varphi$-twists of $T,N$ and $W^*$. We have an exact sequence
$$
1 \rightarrow T_\varphi \rightarrow N_\varphi \rightarrow W_\varphi \rightarrow 1,$$ 
where $W_\varphi$ is the $\varphi$-twist of $W$. Note that $W_\varphi (k)$ is equal to the centralizer $C_W(w_0)$ of $w_0$ in $W$, and similarly $W^*_\varphi (k)$ is equal to $C_{_{W^*}}(w^*_0)$.

As in \S 13.1, let $E$ be the $2$-torsion subgroup of $T_\varphi (k)$. It is normalized by $C_{_{W^*}}(w^*_0)$. Define $A \subset G_\varphi (k)$ to be the group $A = E\cdot C_{_{W^*}}(w^*_0)$. By (ii), we have an exact sequence:
$$
1 \rightarrow E \rightarrow A \rightarrow C_W(w_0) \rightarrow 1\, ,
$$
which shows that $v_2(A) = v_2(E) + v_2\big(C_W(w_0)\big)$. The fact that $w_0$ is regular of order $2$ implies that 
$$|C_W(w_0)| = \prod_{2|d_i} d_i, $$hence $v_2\big(C_W(w_0)\big) = \sum v_2(d_i) = v_2 (W)$. This gives:
\begin{equation}
\label{13.2.1} v_2 (A) = v_2(E) + v_2(W).
\end{equation}
\begin{proposition}We have{\rm{ :}}
$$
\begin{array}{lll}
v_2(E) = r_1 + r_0m & {\mbox{\sl  in case}} \,\,{\rm (b)}\\
v_2(E) = r_0m &{\mbox{\sl in case}} \,\, {\rm (c)}. 
\end{array}\qquad\qquad\qquad\qquad\qquad\qquad\qquad\qquad
$$
\end{proposition}

In case (b), this shows that $(G_\varphi , A)$ is optimal, which proves th.10. Similarly, the fact that $v_2(A) = r_0m + v_2(W)$ proves th.11 in case (c).

% -------------------------------------------------subsection 13.3.----------------------------------------
\subsection{Proof of proposition 6.} We  need to describe explicitly the torus $T_\varphi$. To do so, let us first define the following two tori:

$\G^\sigma_m = 1$-dimensional torus deduced from $\G_m$ by Galois twist relatively to $K/k$.
Its group of $k$-points is $K^*_1 = {\rm Ker} \, N_{K/k} : K^* \rightarrow k^*$.

$R_{K/k}\G_m = 2$-dimensional torus deduced from $\G_m$ by Weil's restriction of scalars relatively to $K/k$. Its group of $k$-points is $K^*$.

\begin{lemma} The torus $T_\varphi$ is isomorphic to the product of $r_1$ copies of \linebreak$R_{K/k}\G_m$ and $r_0-r_1$ copies of $\G^\sigma_m$. 
\end{lemma}

\noindent {\sl Proof.} The character group $X = $ Hom$(T,\G_m)$ is free of rank $r$, with basis the fundamental weights $\omega_1, \dots ,\omega_r$. This gives a decomposition of $T$ as 
$$T = T_1 \times T_2 \times \dots \times T_r\, ,$$
where each $T_i$ is canonically isomorphic to $\G_m$. Let $\tau = -w_0$ be the opposition involution of the root system $R$; it permutes $\omega_1, \dots ,\omega_r$ with $r_1$ orbits of order 2, and $r_0-r_1$ orbits of order 1. (This follows from the fact that $-1$ is an eigenvalue of $w_0$ of multiplicity $r_0$.) The involution $\tau$ permutes the tori $T_j$. If an index $j$ is fixed by $\tau$, then $w_0$ acts on $T_j$ by $t \mapsto t^{-1}$ and the twisted torus $(T_j)_\varphi$ is isomorphic to $\G_m^\sigma$; similarly, if $\tau$ permutes $j$ and $j^\prime$, the torus $(T_j\times T_{j^\prime})_\varphi$ is isomorphic to $R_{K/k}\G_m$. This proves lemma 9. \hfill$\Box$

\vskip 0.3cm
\noindent{\sl End of the proof of prop.6}. The $2$-torsion subgroup of 
${\bf G}_m^\sigma(k) = K^*_1$ is cyclic of order $2^m$; the 2-torsion subgroup of $R_{K/k}\G_m(k) = K^*$ is cyclic of order $2^{m+1}$ in case (b) and of order $2^m$ in case (c). We get what we wanted, namely:

case (b): $v_2(E) = r_1 (m+1) + (r_0-r_1)m = r_0m+r_1$

case (c): $v_2(E) = r_1m + (r_0-r_1)m = r_0m.$

\noindent This completes the proof of prop.6, and hence of th.10 and th.11.\hfill $\Box$

%-----------------------------------------------subsection 13.4---------------------------------------------------

\subsection{Remarks on the non simply connected case.} The proof above could have been given without  assuming that the split group $G$ is simply connected. The main difference is in lemma 9: in the general case, the torus $T_\varphi$ is a product of three factors (instead of two):
$$
T_\varphi = (\G_m)^\alpha \times (\G_m^\sigma)^\beta \times (R_{K/k}\G_m)^\gamma\, ,
$$
where $\alpha , \beta , \gamma$ are integers, with $\beta + \gamma = r_0$ and $\alpha + \gamma = r_1$. This gives the following formulae for $v_2(E):$

case (b) :  $v_2(E) = \alpha + \beta m + \gamma (m+1) = r_1 + r_0m$

case (c)  :  $v_2(E) = \alpha + \beta m + \gamma m = \alpha + r_0m\, .$

\noindent In case (b) one finds the same value for $v_2(A)$, namely the 
M-bound. \linebreak In case (c) one finds a result which is intermediate between the M-bound $r_1 + r_0m + v_2(W)$ and the value $r_0m + v_2(W)$ given by th.11.\medskip

\noindent{\sl Examples} (assuming we are in case (c)).

- {\sl Type} $A_r$, $r$ {\sl even}. One finds that $\alpha$ is always 0, so that one does not gain anything by choosing non simply connected groups. Indeed, in that case, it is possible to prove, by a variant of Schur's method, that the value of $v_2(A)$ given by th.11 is best possible.

- {\sl Type} $A_r$, $r$ {\sl odd} $\ge 3$. Here $r_1 = (r-1)/2$. One finds that $\alpha = 0$ if $r \equiv 1$ (mod 4), but that $\alpha$ can be equal to $1$ if $r \equiv 3$ (mod 4). When $r = 3$, we thus get $\alpha = r_1$; this shows that the M-bound is best possible for type $A_3$.

- {\sl Type} $D_r$, $r$ {\sl odd}. Here $r_1 = 1$, and if one chooses $G$ neither simply connected nor adjoint, one has $\alpha = 1$. This means that the orthogonal group $\SO_{2r}$ has an inner $k$-form which contains an optimal $A$. (Note the case $r = 3$, where $D_3 = A_3$.) 

- {\sl Type} $E_6$. Here $r_1 = 2$, and one has $\alpha = 0$ both for the simply connected group and for the adjoint group (indeed, $\alpha$ is 0 for every adjoint group).\linebreak I do not know whether the bound of th.11 is best possible in this case. 

\vskip 0.5cm
\begin{center}
{\bf {\S 14. The case $m = \infty$}}
\end{center}
\label{sec14}
\setcounter{section}{14}
\setcounter{subsection}{0}

\subsection{Statements.} We keep the notation $(G,R,W,d_i,\ell,t,m)$ of \S 4 and \S 6; as before, we assume that $G$ is of inner type. 

We consider the case $m = \infty$, i.e. the case where {\sl the image of} 
$\chi_{_{\ell^\infty}}$ {\sl is finite}; that image is then cyclic of order $t$, cf. \S 4.

Let $a(t)$ be the number of $i$'s such that $d_i \equiv 0$ (mod $t$). If $a(t) = 0$, then $G(k)$ is $\ell$-torsion free, cf. \S 6.2, cor.to prop. 4. In what follows, we shall thus assume that $a(t) \ge 1$. In that case, $G(k)$ may contain infinite $\ell$-subgroups (we say that a group is an $\ell$-{\sl group} if every element of that group has order a power of $\ell$). The following two theorems show that $a(t)$ controls the size of such a subgroup:

\begin{mytheorem12} Let  $A$ be an $\ell$-subgroup of $G(k)$. Then $A$ contains a subgroup of finite index isomorphic to the $\ell$-group $(\Q_\ell/\Z_\ell)^a$, with $a \le a(t).$ 
\end{mytheorem12}

(Note that $\Q_\ell/\Z_\ell$ is the union of an increasing sequence of cyclic groups of order $\ell, \ell^2$, \dots ; it is the analogue of $\Z/\ell^m\Z$ for $m = \infty$.)\vskip.1cm

The bound $a \le a(t)$ of th.12. is optimal. More precisely:

\begin{mytheorem13} There exist a semisimple group $G$ of inner type, with root system $R$, and an $\ell$-subgroup $A$ of $G(k)$, such that $A$ is isomorphic to the product of $a(t)$ copies of $\Q_\ell/\Z_\ell$.
\end{mytheorem13}

\subsection{Proof of theorem 12} We need a few lemmas:

\begin{lemma} Any finitely generated $\ell$-subgroup of $G(k)$ is finite.
\end{lemma}

\noindent{\sl Proof.} Let $B$ be a finitely generated $\ell$-subgroup of $G(k)$. We may embed $B$ in $\GL_n(k)$ for $n$ large enough. By a known result (see \S 1.2) there exists a subgroup $B^\prime$ of $B$, of finite index, which is torsion-free  if $\car(k)=0$, and has only $p$-torsion if char$(k)$ = $p$. Since $B^\prime$ is an $\ell$-group, this means that $B^\prime = 1$, hence $B$ is finite.\hfill$\Box$

\begin{lemma}
There exists a maximal $k$-torus of $G$ which is normalized by $A$. {\rm (Recall that $A$ is an $\ell$-subgroup of $G(k).)$}
\end{lemma}

\noindent {\sl Proof.} Let $F$ be the set of all finite subgroups of $A$, ordered by inclusion. Lemma 10 implies that, if $B_1$ and $B_2$ belong to $F$, so does $\<B_1, B_2\>$. Let $X$ be the $k$-variety parametrizing the maximal tori of $G$; it is a homogeneous space of $G$. If $B \in F$, let $X^B$ be the subvariety of $X$ fixed by $B$; a point of $X^B$ corresponds to a maximal torus of $G$ normalized by $B$. By the noetherian property of the scheme $X$, one may choose $B_0 \in F$ such that $X^{^{B_0}}$ is minimal among the $X^B$\ \!'s. If $B \in F$, then $X^{\<B_0,B\>}$ is contained in $X^{B_0}$, hence equal to $X^{B_0}$. This shows that $X^{B_0}$ is contained in all the \linebreak $X^B$\ 's, i.e. that every maximal torus which is normalized by $B_0$ is normalized by all the $B$'s, hence by $A$. By the corollary to th.3$^{\prime\prime}$ of \S 3.3 (applied to the finite $\ell$-group $B_0$) there exists such a torus which is defined over $k$. \hfill $\Box$

\begin{lemma}
Let $u\in \M_r(\Z_\ell)$ be an $r \times r$ matrix with coefficients in $\Z_\ell$, which we view as an endomorphism of $(\Q_\ell/\Z_\ell)^r$. Then {\rm Ker}$(u)$ has a subgroup of finite index isomorphic to the product of $r- {\rm rank}(u)$ copies of $\Q_\ell/\Z_\ell$.
\end{lemma}

In other words, the ``corank" of Ker$(u)$ is equal to $r - {\rm rank}(u)$.
\vskip 0.3cm
\noindent{\sl Proof}. Same as that of lemma 4 of \S 5.2: by reduction to the case where $u$ is a diagonal matrix.\hfill$\Box$

\begin{lemma}
Let $z_t$ be a primitive $t$-th root of unity, and let $w$ be an element of $W$. The multiplicity of $z_t$ as an eigenvalue of $w$ is $\le a(t)$.
\end{lemma}

\noindent{\sl Proof.} See \cite{Sp74}, th.3.4(i) where it is deduced from the fact that the polynomial $\det(t-w)$ divides $\prod_i (t^{d_i}-1).$\hfill$\Box$

\begin{lemma}
Let $T$ be a maximal $k$-torus of $G$, and let $T(k)_\ell$ be the $\ell$-torsion subgroup of $T(k)$. We have  {\rm corank} $T(k)_\ell \le a(t)$.
\end{lemma}

As above, the ``corank" of a commutative $\ell$-group is the largest $n$ such that the group contains the product of $n$ copies of $\Q_\ell/\Z_\ell$.

\vskip 0.3cm
\noindent{\sl Proof.} As in \S 5.2, let $Y(T) = \Hom_{k_s} (\G_m,T)$ be the group of cocharacters of $T$.  The action of the Galois group $\Gamma_k$ on $Y(T)$ gives a homomorphism
$$
\rho : \Gamma_k \rightarrow \Aut Y(T) \simeq \GL_r (\Z)
$$
and the image of $\rho$ is contained in the Weyl group $W$ (this is still another way of saying that $G$ is of inner type). The group $\Gamma_k$ acts on $T(k_s)_\ell \simeq (\Q_\ell/\Z_\ell)^r$ by $\rho \otimes \chi$, where $\chi = \chi_{_{\ell^\infty}}$. Let us now choose $g \in \Gamma_k$ such that $\chi(g) = z_t^{-1}$, where $z_t$ is an element of order $t$ of $\Z^*_\ell$, and let $w = \rho (g)$. The element $g$ acts on $T(k_s)_\ell$ by $wz^{-1}_t$. Let $T_g$ be the kernel of $g-1$ on $T(k_s)_\ell$. By lemma 12, we have corank $(T_g) = r -$ rank$(g-1)$, which is equal to the multiplicity of $z_t$ as an eigenvalue of $w$; using lemma 13, we get corank$(T_g) \le a(t)$, and since $T(k)_\ell$ is contained in $T_g$, we have corank$(T(k)_\ell) \le a(t)$. \hfill$\Box$

\vskip 0.3cm
\noindent{\sl End of the proof of th.12}. By lemma 11, there is a maximal $k$-torus $T$ of $G$ which is normalized by $A$. Let $A^\circ = A \cap T(k)$. Then $A^\circ$ is an abelian subgroup of $A$ of finite index. Since $A^\circ$ is contained in $T(k)_\ell$, lemma 14 shows that $A^\circ$ is isomorphic to the product of a finite group with a product of at most $a(t)$ copies of $\Q_\ell/\Z_\ell$. \hfill$\Box$

\subsection{Proof of  theorem 13.} We follow the same strategy as for theorem 9, 10 and 11. There are three cases:

\subsubsection{\rm{\bf Classical groups} $(\ell \not= 2)$} We change slightly the definitions of \S 9.1: we define $A_N$ as the subgroup of $\GL_N(K)$, with $K = k(z_\ell)$, made up of the diagonal matrices of order a power of $\ell$; it is isomorphic to $(\Q_\ell/\Z_\ell)^N$. 

For any given $n \ge 2$, we put $N = [n/t]$ and we get embeddings 
$$
A_N \rightarrow \GL_N(K) \rightarrow \GL_{Nt}(k) \rightarrow \GL_n (k).
$$
If $t>1$, one checks that the $k$-determinant of every element of $A_N$ is $1$; we thus get an embedding $A_N \rightarrow \SL_n(k)$ which has the required properties since $N = a(t)$ in that case. When $t = 1$, we replace $A_N$ by the subgroup of its elements of $k$-determinant 1, and we also get what we want. This solves the case of type $A_r$. Types $B_r$, $C_r$ and $D_r$ are then treated by the methods of \S 9.3 and \S 9.4.

\subsubsection{{\rm\bf Exceptional groups} $(\ell \not= 2)$} One replaces prop.5 of \S 11.4 by a statement giving the existence of $A \subset G_\varphi(k)$ with $A\simeq(\Q_\ell/\Z_\ell)^{a(t)}$. The proof is the same. One then proceeds as in \S 12.

\subsubsection{{\rm\bf The case} $\ell = 2$} Same method as in \S 13.\hfill$\Box$

%\end{proof}

\newpage
\markright{ }

\vfill
 \noindent 
{\small \noindent J.-P. Serre\\
\noindent Coll\`{e}ge de France\\
\noindent 3, rue d'Ulm\\
\noindent F-75005 PARIS.}

\end{document}